\documentclass[11pt]{amsart}
\usepackage{amssymb, amsmath, amsmath}
\usepackage[backref]{hyperref}
\usepackage{mathrsfs}
\usepackage[alphabetic,initials,backrefs,lite]{amsrefs}
\usepackage{amscd}   
\usepackage{fullpage}
\usepackage[all]{xy} 
\usepackage{setspace}
\usepackage[margin=2.2  cm, headsep = 0.1 cm, footskip = 0.7 cm]{geometry}

\DeclareFontEncoding{OT2}{}{} 

\usepackage[usenames,dvipsnames]{color}

\newtheorem{lemma}{Lemma}[section]

\newtheorem{prop}[lemma]{Proposition}

\newtheorem{conj}[lemma]{Conjecture}

\newtheorem{question}[lemma]{Question}
\newtheorem{claim*}{Claim}
\newtheorem{thm}[lemma]{Theorem}
\newtheorem*{thm*}{Theorem}

\theoremstyle{definition}
\newtheorem{defn}[lemma]{Definition}
\newtheorem{remark}[lemma]{Remark}

\newcommand{\A}{{\mathbb A}}
\newcommand{\G}{{\mathbb G}}

\newcommand{\PP}{{\mathbb P}}

\newcommand{\Q}{{\mathbb Q}}

\newcommand{\Z}{{\mathbb Z}}

\newcommand{\Xbar}{{\overline{X}}}

\newcommand{\kbar}{{\overline{k}}}

\newcommand{\Ybar}{{\overline{Y}}}

\newcommand{\Adeles}{{\mathbb A}}

\newcommand{\calB}{{\mathcal B}}

\newcommand{\calG}{{\mathcal G}}

\newcommand{\calO}{{\mathcal O}}

\newcommand{\calU}{{\mathcal U}}
\newcommand{\calV}{{\mathcal V}}

\newcommand{\calY}{{\mathcal Y}}

\newcommand{\frakF}{{\mathfrak F}}

\newcommand{\frakR}{{\mathfrak R}}

\DeclareMathOperator{\supp}{supp}

\DeclareMathOperator{\inv}{inv}

\DeclareMathOperator{\im}{im}

\DeclareMathOperator{\Hom}{Hom}

\DeclareMathOperator{\Gal}{Gal}

\DeclareMathOperator{\Br}{Br}
\DeclareMathOperator{\etBr}{et,Br}

\DeclareMathOperator{\Sym}{Sym}

\DeclareMathOperator{\Pic}{Pic}

\DeclareMathOperator{\Spec}{Spec}

\DeclareMathOperator{\et}{et}

\DeclareMathOperator{\res}{res}
\DeclareMathOperator{\cores}{cores}

\DeclareMathOperator{\CH}{CH}

\newcommand{\isom}{\simeq}

\newcommand{\into}{\hookrightarrow}

\numberwithin{equation}{section}
\numberwithin{table}{section}

\newcommand{\defi}[1]{\textsf{#1}} 

\def\AA{\mathbb{A}}

\def\NN{\mathbb{N}}

\def\PP{\mathbb{P}}

\def\calB{\mathcal{B}}

\def\calG{\mathcal{G}}

\def\calO{\mathcal{O}}

\def\calU{\mathcal{U}}
\def\calV{\mathcal{V}}

\def\calY{\mathcal{Y}}

\newcommand\xdownarrow[1][3ex]{%
   \mathrel{\rotatebox[origin=c]{90}{$\xleftarrow{\rule{#1}{0pt}}$}}
}

\usepackage{tikz-cd}

\usepackage{comment}
\usepackage{enumitem}
\usepackage{microtype}

\newenvironment{talign*}
 {\let\displaystyle\textstyle\csname align*\endcsname}
 {\endalign}
 
\allowdisplaybreaks

\title{Descent and étale-Brauer obstructions for 0-cycles}
\author{Francesca Balestrieri}
\address{The American University of Paris,  5 Boulevard de La Tour-Maubourg, 75007 Paris, France}
\email{fbalestrieri@aup.edu}
\author{Jennifer Berg}
\address{Bucknell University, Department of Mathematics, Lewisburg, PA 17837, USA}
\email{jsb047@bucknell.edu}


\begin{document}

\maketitle
\begin{abstract}
For 0-cycles on a variety over a number field, we define an analogue of the classical descent set for rational points. This leads to, among other things, a definition of the \'etale-Brauer obstruction set for 0-cycles, which we show is contained in the Brauer-Manin set and is compatible with Suslin's singular homology of degree 0. We then transfer some tools and techniques used to study the arithmetic of rational points into the setting of 0-cycles. For example, we extend the strategy developed by Y. Liang, relating the arithmetic of rational points over finite extensions of the base field to that of 0-cycles, to torsors. We give applications of our results to study the arithmetic behaviour of 0-cycles for  Enriques surfaces,  torsors given by (twisted) Kummer varieties, universal torsors, and torsors under tori. 
\end{abstract}

\section{Introduction}
\subsection{Obstruction sets for rational points.} Let $k$ be a number field and $X$ be a smooth, proper, geometrically integral variety over $k$. In order to investigate questions concerning the qualitative arithmetic behaviour of the set $X(k)$ of rational points on $X$, a well-known strategy is that of defining so-called \emph{obstruction sets} -- namely, certain subsets of the set $X(\A_k)$ of adelic points on $X$ that still contain $X(k)$ -- and trying to exploit the more tractable local nature of these sets to study the specific question at hand. Such sets can be used, for example, to determine whether $X(k) = \emptyset$ by refining local-to-global principles, or to classify varieties according to their arithmetic behaviour. The theory of obstruction sets in the context of rational points has been developed quite extensively over the last several decades. In \cites{Man71}, Manin first constructed an obstruction set by using the Brauer group $\Br(X) := H^2_{\et}(X, \G_m)$ and class field theory to define what is now known as the \emph{Brauer-Manin set}, namely 
\[ X(\Adeles_k)^{\Br} := \bigcap_{\alpha \in \Br X}\left \{ (x_v)_{v \in \Omega_k} \in X(\Adeles_k) : \textstyle \sum_{v \in \Omega_k} \inv_v \alpha(x_v) = 0 \right \}. \]
Later, Colliot-Thél\`ene and Sansuc \cites{CTS87} defined a new type of obstruction sets known as \emph{descent sets}, based not on the Brauer group, but rather on the notion of torsors under algebraic groups. That is, if $G$ is a linear algebraic group over $k$ and if $[g: Y \to X] \in H^1_{\et}(X,G)$ is (the $k$-class of) a $G$-torsor over $X$, then we can define the \emph{descent set associated to $g$} as
\[ X(\A_k)^g := \bigcup_{[\tau] \in H^1(k,G)} g^\tau(Y^\tau(\Adeles_k)), \]
where the $G^\tau$-torsor $g^\tau: Y^\tau \to X$ is the twist of the torsor $g: Y \to X$ by $\tau$ (see \cite{Sko01}*{Ch.\ 2}).
We can also combine these two main types of obstructions together to yield potentially finer obstruction sets. For example, the \emph{\'etale-Brauer} set,
\[X(\Adeles_k)^{\et, \Br} := \bigcap_{F \textrm{ finite lin alg $k$-group}} \enskip \bigcap_{[f: Y \to X] \in H^1_{\et}(X,F)} \enskip \bigcup_{\tau \in H^1_{\et}(k,F)} f^\tau(Y^\tau(\A_k)^{\Br})),\] 
can be obtained by considering the Brauer-Manin sets of all finite \'etale covers of $X$. Although the obstruction sets defined using the Brauer group and those defined by using torsors have quite different natures, there is sometimes a close 
interrelation between them -- see, for example \cite{Sko01}, \cite{Har02}, \cite{Sto07}, \cite{Dem09}, \cite{Skoro09}, \cite{Bal16}, \cite{Cao20}. 
Unfortunately, despite the richness of the different obstructions available, the arithmetic behaviour of rational points on varieties is still not completely understood in general: for example,  in \cites{Poo10} Poonen constructed a variety $X$ over a number field $k$ such that $X(\A_k)^{\et \Br} \neq \emptyset$ but $X(k) = \emptyset$, showing that even the finest obstruction set currently at our disposal, i.e., the étale-Brauer set, is not quite refined enough to capture the lack of rational points.

\subsection{Obstruction sets for 0-cycles.} Let us now consider the theory of obstruction sets in a different context, namely, that of 0-cycles on $X$, which can be viewed as generalisations of rational points. Given a fixed integer $d$, the main object of interest is the set $Z_0^d(X)$ of 0-cycles of degree $d$ on $X$, that is, the set of all formal $\Z$-sums $z:= \sum_{x} n_x x$ of closed points $x$ on $X$ such that $\deg(z) := \sum_{x \in X} n_x [k(x):k] = d$, where $k(x)$ is the residue field of $x$.  Much in the same way as we did for rational points, we can ask qualitative arithmetic questions about $Z_0^d(X)$ -- for example, whether $Z_0^d(X)= \emptyset$. The basic strategy to tackle such questions remains the same, that is, we consider subsets of the set $Z_0^d(X_{\A_k})$ of adelic degree $d$ 0-cycles on $X$ and try to exploit the local nature of these subsets to draw conclusions about $Z_0^d(X)$.

There are two striking differences between the theory of obstructions for rational points and that for 0-cycles, however. First, while there is a wide range of obstructions sets and tools currently available for rational points, the same cannot be said for 0-cycles. One key complication is that, in the definition of 0-cycles, several different field extensions need to be considered at once, thereby creating challenges in generalising obstruction sets from rational points to this new context. Currently, the only obstruction set that has been generalised is the Brauer-Manin obstruction \cites{CT95}, defined in a similar way as in the context of rational points, but with the extra use of corestriction maps to deal with the different residue fields of the support of the 0-cycles, namely
{\small
\[ Z_0^d(X_{\A_k})^{\Br}:= \bigcap_{\alpha \in \Br X} \left \{ \left(\sum_{x_v \in X_{k_v}} n_{x_v} x_v\right)_{v \in \Omega_k } \in Z_0^d(X_{\A_k}) :  \sum_{v \in \Omega_k }\sum_{x_v \in X_{k_v}} n_{x_v} \inv_v\left( \cores_{k_v(x_v)/k_v} \alpha(x_v)\right) = 0\right\}.\]
}

The second difference concerns the arithmetic behaviour of $0$-cycles. While, as previously mentioned, even the finest known obstruction cannot explain all failures of the Hasse principle for rational points, the qualitative arithmetic of 0-cycles is, conjecturally, completely captured by the Brauer-Manin obstruction -- and thus  much more well behaved.
Indeed, Colliot-Th\'el\`ene \cite{CT95} conjectured that the Brauer-Manin obstruction is the only one to weak approximation for $0$-cycles of degree $1$ on any smooth, proper, geometrically integral variety $X$ over $k$ (see also \cite{KS86}). Colliot-Th\'el\`ene's conjecture is encompassed by the following main guiding conjecture in the context of $0$-cycles, which is known for a few cases (see \cite{Sal88}, \cite{Sai89}, \cite{CT99}, \cite{ES08}, \cite{Wit12}, \cite{Lia13}, and \cite{CTS21}*{p.\ 383-4} for a detailed list of references) but remains open in general.  

\begin{conj}[Conjecture (E)] \label{conjE}   For any smooth, projective variety $X$ over a number field $k$, the following complex is exact,
\[ \widehat{\CH_0(X)} \longrightarrow \prod_{v \in \Omega_k} \widehat{\CH_0'(X_{k_v})} \longrightarrow \Hom(\Br X, \Q/\Z),\]
where $\widehat{\ \ \ \ }$ denotes the profinite completion, and where $\CH_0'(X_{k_v})$ coincides with $\CH_0(X_{k_v})$ at finite places and is a modification of the Chow group at infinite places (see \cite{Wit12}*{\S1.1} or \cite{CT95}*{\S1} for further details).
\end{conj}

\subsection{Main results}
The main aim of this paper is to 
fill in some significant gaps in the theory of obstructions for 0-cycles by defining an analogue of the descent set associated to a torsor for $0$-cycles. 
Given a fixed $d \in \Z$ and a fixed torsor $g: Y \to X$ under a linear algebraic group $G$ over $k$,  we define the \emph{degree $d$ descent set associated to $g$} (see Definition \ref{def1}) to be
{\small \[ Z_0^d(X_{\A_k})^g :=
\bigcup_{\substack{S \textrm{ finite non-empty}\\\textrm{ set of finite field}\\
\textrm{extensions of $k$}} }  \textbf{g}_{\ast, \textrm{ad}} \left(  \coprod_{\substack{(T_L \subset H^1(L,G))_{L \in S}: \\ \textup{$T_L\neq \emptyset$ finite for all $L \in S$}}}  \coprod_{\substack{\left((\Delta_\tau)_{\tau \in T_L}\right)_{L \in S}: \\ \Delta_\tau \in \Z \textrm{ for all $\tau$}, \\ \sum_{L \in S} \sum_{\tau \in T_L} \Delta_\tau [L:k]= d}} \prod_{L \in S} \prod_{\tau \in T_L} Z_0^{\Delta_\tau}(Y_{\A_{L}}^{\tau})  \right),\]}
where
{\small \[  \textbf{g}_{\ast, \textrm{ad}}:  
 \left(\left( \left( \left( \sum_{y \in Y_{L_w}^{\tau}} n_{y}\; y \right)_{\substack{w \in \Omega_{L} :\\ w|v}}\right)_{\tau \in T_L} \right)_{L \in S} \right)_{v \in \Omega_k} \longmapsto   \left( \sum_{L \in S}\sum_{\tau \in T_L} \sum_{w \in \Omega_L: w \mid v} \sum_{y \in Y_{L_w}^{\tau}} n_{y} (s_{L_w} \circ g_{L_w}^{\tau})_\ast (y)\right)_{v \in \Omega_k}\]}
 and where $g_{L_w}^{\tau} \colon Y_{L_w}^{\tau} \to X_{L_w}$ and $s_{L_w} \colon X_{L_w} \to X_{k_v}$ are the natural maps. This set $Z_0^d(X_{\A_k})^g$ contains $Z_0^d(X)$ and, moreover, generalises the descent set $X(\A_k)^g$ when $d=1$. This set arises in a very natural way that mimics the way in which the descent set for rational points is defined, once we take into account the more complex structure of 0-cycles (see Sections \ref{sec: global} and \ref{sec: local} for more  details).

The definition of the descent set leads to, among other results, a definition of the \'etale-Brauer obstruction for $0$-cycles (Definition \ref{def:fBrSet}), namely
\[ \small Z_0^d(X_{\A_k})^{ \etBr} := \! \! \! \! \bigcup_{\substack{S \neq \emptyset \textrm{ finite set}\\\textrm{ of finite field}\\ \textrm{ extns of $k$}}}   \ \bigcap_{\substack{G \text{ finite lin.}\\\text{alg.\ $k$-group}}} \, \bigcap_{\substack{[f: Y \to X]\\ \in H^1(X,G)}}  \textbf{f}_{\ast, \textrm{ad}} \left(  \coprod_{\substack{(T_L \subset H^1(L,G))_{L \in S} \\ \textup{$T_L\neq \emptyset$ finite} \\ \text{for all $L \in S$}}}  \coprod_{\substack{\left((\Delta_\tau)_{\tau \in T_L}\right)_{L \in S}: \\ \Delta_\tau \in \Z \textrm{ for all $\tau$}, \\ \sum_{L \in S} \sum_{\tau \in T_L} \Delta_\tau [L:k]= d}} \prod_{L \in S} \prod_{\tau \in T_L} Z_0^{\Delta_\tau}(Y_{\A_{L}}^{\tau})^{\Br} \right).\]
In Section \ref{sec: local}, we show how the descent set and its variations involving the Brauer group are compatible with an important equivalence relation on 0-cycles, namely the finite correspondences equivalence relation, which gives rise to Suslin's singular homology groups of degree 0; for proper varieties, Suslin's singular homology groups of degree 0 correspond to Chow groups.

More generally,  these definitions afford the possibility of a more systematic theory of obstructions for $0$-cycles, analogous to that of rational points. It also provides new potential avenues for investigating Conjecture \ref{conjE}: if Conjecture \ref{conjE} holds, it implies that the Brauer-Manin obstruction is the only one for weak approximation for 0-cycles of degree 1 (see \cite{Lia13}*{Theorem A}); this, on the other hand, should imply that the étale-Brauer set and the Brauer-Manin set for 0-cycles of degree 1 are ``equal" in the Chow group, in the sense that, for any $n \in \NN$, for any finite set $S' \subset \Omega_k$ of places of $k$, if $(z_v)_v \in Z_0^1(X_{\A_k})^{\Br}$, then there is some $(\tilde{z}_v)_v \in Z_0^1(X_{\A_k})^{\etBr}$ such that $z_v$ and $\tilde{z}_v$ have the same image in $\CH_0(X_{k_v})/n$ for any $v \in S'$.

Equipped with our new definitions, in \S\ref{sec:LiangTorsors} we begin to transfer some of the tools and techniques used to study descent obstructions to rational points into the $0$-cycles setting.  We do so in the spirit of some of the ideas developed by Liang in \cite{Lia13}. His strategy consists of trying to show that, under certain conditions, if some arithmetic property (such as, for example, the Brauer-Manin obstruction being the only one for weak approximation) holds for rational points for enough field extensions of finite degree of the base field, then an analogous arithmetic property holds for 0-cycles as well. 
In \S\ref{subsec:LiangTorsors}, we extend Liang's strategy to torsors.

\begin{thm*}[Theorem \ref{thm: fBr}] Let $X$ be a smooth, proper, geometrically integral variety over a number field $k$. Let $f : Y \to X$ be an $F$-torsor for some linear algebraic group $F$ over $k$ and with $Y$ geometrically integral. Let $d$ be any integer. Assume that
\begin{enumerate}[label = (\roman*)]
\item for any finite extension $K/k$ and any $\tau \in H^1(K, F)$, the quotient $\Br_{nr}(Y^\tau_{K})/\Br_0 (Y^\tau_{K})$ is finite, and there exists a finite extension $K'_\tau$ of $K$ so that for all finite extensions $L$ of $K$ linearly disjoint from $K'_\tau$ over $K$, the homomorphism induced by restriction 
\[ \res_{L/K} : \Br_{nr}(Y^\tau_{K})/\Br_0 (Y^\tau_{K}) \to \Br_{nr}(Y^\tau_L)/\Br_0(Y^\tau_L)\]
is surjective;
\item  for any finite extension $L/k$, we have that $X_L(\Adeles_L)^{f_L, \Br_{nr}} \neq \emptyset $ if and only if $ X(L) \neq \emptyset$ (respectively, if $X_L(\Adeles_L)^{f_L, \Br_{nr}} \ne \emptyset$, then weak approximation holds for $X_L$).

\end{enumerate}
Then $f$-descent with unramified Brauer obstruction is the only obstruction to the Hasse principle (respectively, weak approximation) for 0-cycles of degree $d$ on $X$.
\end{thm*}

In \S\ref{sec: applications} we consider several applications of our tools, often in contexts in which there is deeper knowledge of the obstructions that govern the existence and density of rational points. 
Firstly, building on work by Ieronymou \cite{Ier21}, we get an  application of the \'etale-Brauer set for Enriques surfaces -- at least, conditionally on a conjecture by Skorobogatov. 

\begin{thm*}[Theorem \ref{thm: etBrEnriques}] Let $X$ be an Enriques surface over a number field $k$ and let $f: Y \to X$ be a K3 covering of $X$, i.e. a $\Z/2\Z$-torsor over $X$ with $Y$ a K3 surface. Let $d \in \Z$. Assume that Conjecture \ref{conj: Skoro} holds. Then, for any positive integer $n$, if $ (z_v)_v \in Z_0^d(X_{\A_k})^{f, \Br}$ then there exists a global 0-cycle $z_{n} \in Z_0^d(X)$ such that $z_{n}$ and $(z_v)_v$ have the same image in $\CH_0(X_{k_v})/n$ for all $v \in \Omega_k$. 
\end{thm*}

As a further application, we study the arithmetic of 0-cycles in the context of what can be considered as higher-dimensional generalisations of Enriques surfaces, where we look at the case where our torsors are (twisted) Kummer varieties.

\begin{thm*}[Theorem \ref{thm: Kummer}] Let $X$ be a smooth, proper, geometrically integral variety over a number field $k$. Let $f: Y \to X$ be a torsor under some linear algebraic group $F$ over $k$, where $Y$ is a (twisted) Kummer variety over $k$. Let $d \in \Z$ be \emph{odd}. Assume that Question \ref{q:kum} has a positive answer. Then
$Z_0^d(X_{\A_k})^{f,\Br\{2\}} \neq \emptyset$ implies $Z_0^{d}(X) \neq \emptyset$.
\end{thm*}
We remark that the recent preprint \cite{Ier22} by Ieronymou should remove both the condition that $d$ be odd and the need to restrict to the 2-primary part of the Brauer group. We nonetheless give the theorem in this form, as its proof can potentially be applied to other situations.

Finally, we consider universal torsors and torsors under tori. In the rational points setting, it is well-known that, if a universal torsor $g: W \to X$ exists, then $X(\A_k)^g = X(\A_k)^{\Br_1}$ \cite{Sko01}*{Lemma 2.3.1}. In Theorem~\ref{univt}, we recover an analogous result for $0$-cycles.

\begin{thm*} [Theorem \ref{univt}] 
Let $X$ be a smooth, proper, geometrically integral variety over $k$ with $\Pic \Xbar$ finitely generated as a $\Z$-module. Suppose that a universal torsor $g: W \to X$ under some group $G$ of multiplicative type over $k$ exists. Then, for any integer $d \in \Z$, for any positive integer $n$, and for any finite subset $S' \subset \Omega_k$ of places of $k$, we have that
\begin{enumerate}
\item if $(z_v)_v \in \Z_0^d(X_{\A_k})^g $, then there exists some $(u_v)_v \in Z_0^d(X_{\A_k})^{\Br_1}$ such that $z_v$ and $u_v$ have the same image in $\CH_0(X_{k_v})/n$ for all $v \in S'$;
\item if, moreover, $\Br_1(X)/\Br_0(X)$ is finite, then $(z_v)_v \in \Z_0^d(X_{\A_k})^{\Br_1}$ implies that there exists some  $(u_v)_v \in \Z_0^d(X_{\A_k})^{g}$ such that $z_v$ and $u_v$ have the same image in $\CH_0(X_{k_v})/n$ for all $v \in S'$.
\end{enumerate}
\end{thm*}

Similarly, for torsors under tori we obtain the following result, in the spirit of a result by Harpaz and Wittenberg (see \cite{HW}*{Th\'eor\`eme 2.1}).

\begin{thm*}[Theorem \ref{hw0cyc}]
Let $X$ be a smooth, proper, geometrically integral variety over $k$. Let $f: Y \to X$ be a torsor under a $k$-torus $T$. Assume that $\Br X/\Br_0 X$ is finite and that there is some finite extension $F/k$ such that $\res_{l/k} : \Br X/ \Br_0 X \to \Br(X_l)/ \Br_0(X_l)$ is surjective for all finite extensions $l/k$ linearly disjoint from $F$ over $k$. Then, for any integer $d \in \Z$, for any positive integer $n$, and for any finite subset $S' \subset \Omega_k$ of places of $k$, we have that  $(z_v)_v \in \Z_0^d(X_{\A_k})^{\Br}$ implies that there exists some  $(u_v)_v \in \Z_0^d(X_{\A_k})^{f, \Br_{nr}}$ such that $z_v$ and $u_v$ have the same image in $\CH_0(X_{k_v})/n$ for all $v \in S'$.
\end{thm*}

\subsection{Notation and terminology}
Let $k$ be a number field and let $\kbar$ denote a fixed algebraic closure of $k$; we will take any finite extension of $k$ to be inside $\kbar$. Let $\Omega_k$ denote the set of non-trivial places of $k$ and $\A_k$ the ring of adeles of $k$. For a $k$-scheme $X$, write $X_K := X \times_{\Spec k} \Spec K$ for the base change of $X$ to the extension $K/k$, and $\Xbar := X_{\kbar}$. A variety over $k$ is defined as a separated scheme of finite type over $k$. For a closed point $x$ of a $k$-scheme $X$, let $k(x)$ denote the residue field (over $k$). The Brauer group $\Br(X) := H^2_{\et}(X,\G_m)$ of a $k$-variety $X$ is equipped with a natural filtration, $\Br_0(X) \subset \Br_1(X) \subset \Br(X)$, where $\Br_0(X) := \im(\Br(k) \to \Br(X))$ and $\Br_1(X) := \ker(\Br X \to \Br \Xbar)$.
For an abelian group $A$ and an integer $d > 0$, we let $A[d]$ denote the $d$-torsion subgroup of $A$. When $d$ is prime, let $A\{d\}$ denote the $d$-primary part $\varprojlim_n A[d^n]$ of $A$.

\section{A partition for global 0-cycles using torsors} \label{sec: global}
Let $X$ be a variety over a number field $k$. Throughout this section, we fix $g \colon Y \to X$ to be a $G$-torsor for some linear algebraic group $G$ over $k$ and a $Y$ smooth $k$-variety. In this section, we generalise the standard partition of $X(k)$ using the torsor $g: Y \to X$ to 0-cycles. This partition will allow us, in Section \ref{sec: local}, to define the descent set $Z_0^d(X(\A_k))^g$ and the étale-Brauer set $Z_0^d(X(\A_k))^{\et, \Br}$ \ for 0-cycles, analogous to the corresponding sets $X(\A_k)^g$ and $X(\A_k)^{\et,\Br}$ for rational points.

Before we are able to give the definition of the descent set for 0-cycles (see Definition \ref{def1}), let us try to motivate a bit how the construction arises by looking first at what happens for rational points. For rational points, a standard result from the theory of torsors (see e.g. \cite[p.22]{Sko01}) tells us that, using the torsor $g: Y \to X$,  the set of $k$-rational points of $X$ can be partitioned as
\begin{equation} \label{partition} X(k) = \coprod_{\tau \in H_{\et}^1(k,G)} g^\tau (Y^\tau(k)),\end{equation}
where $g^\tau : Y^\tau \to X$ is the $G^\tau$-torsor over $X$ obtained by twisting $g: Y \to X$ by $\tau$. Since $Y^\tau(k) \subset Y^\tau(\A_k)$, we then get
\[ X(k) =\coprod_{\tau \in H_{\et}^1(k,G)} g^\tau (Y^\tau(k)) \subset  \bigcup_{\tau \in H^1_{\et}(k,G)} g^\tau(Y^\tau(\A_k)) =:X(\A_k)^g,\]
where $X(\A_k)^g$ is the $g$-descent set. The main goal of this section is to generalise the partition \eqref{partition} in the context of 0-cycles.
In Definition \ref{def}, we construct a set $Z_0^d(X)^g$ that is the analogue for 0-cycles of the set $\coprod_{\tau \in H_{\et}^1(k,G)} g^\tau (Y^\tau(k))$ in \eqref{partition}: in Proposition \ref{prop:partition}, we indeed prove that  $Z_0^d(X) = Z_0^d(X)^g$, thus yielding the analogue for global 0-cycles of the partition \eqref{partition}. In  Section \ref{sec: local}, we then exploit this partition for 0-cycles in order to define the $g$-descent $Z_0^d(X_{\A_k})^g$ for 0-cycles (see Definition \ref{def1}). 

The natural motivation behind the construction of the set $Z_0^d(X)^g$ -- which, at first sight, might appear a bit abstract -- lies completely in the proof of Proposition \ref{prop:partition}: given a 0-cycle $z \in Z_0^d(X)$, the points in $\supp(z)$ are rational points over their respective residue fields; hence, by using the partition \eqref{partition} and grouping together the points in $\supp(z)$ with a same residue field $L$ and a same pullback class $\tau \in H^1_{\et}(L,G)$ of $[g: Y \to X] \in H^1(X,G)$, we can decompose $z$ as a collection of ``scattered" 0-cycles, whose degrees are compatible in a certain explicit way, living on (potentially) different $Y^\tau_L$'s, for some extensions $L/k$ and $\tau \in H_{\et}^1(L, G)$ depending on $\supp(z)$ only. We can then use the ``recombining" map $\textbf{g}_\ast$ (defined in Definition \ref{def:recombining}) to recombine these ``scattered" 0-cycles into $z$. This shows that any 0-cycle in $Z_0^d(X)$ is in $Z_0^d(X)^g$, and yields the very natural justification for all the objects and conditions appearing in the definition of $Z_0^d(X)^g$.

Let us now delve into the construction.

\begin{defn}
Let $k$ be a number field and let $\kbar$ be a fixed separable closure of $k$.
We let 
\[ \mathfrak{F}_{k} := \{ K \subset \kbar: \textrm{ $K$ is a finite field extension of $k$}\}.\]
\end{defn}

Recall that if $z := \sum n_y y \in Z_0(Y)$, the \emph{pushforward} of $z$ is \[ g_\ast(z) = \sum_{y \in Y} n_y g_\ast(y)  = \sum_{y \in Y} n_y [k(y) : k(g(y))] g(y) \in Z_0(X).\] 

\begin{defn} \label{def:recombining} We define the ``recombining" map 
\[ \textbf{g}_\ast \colon  \coprod_{\substack{S \subset \frakF_k\\ \textup{$S\neq \emptyset$ finite}} }  \coprod_{\substack{(T_L \subset H^1(L,G))_{L \in S} \\ \textup{$T_L\neq \emptyset$ finite for all $L \in S$}}}  \prod_{L \in S} \prod_{\tau \in T_L} Z_0(Y^\tau_L) \to Z_0(X) \]
as follows. Let 
\[z := \left(\left(\sum_{y \in Y_{L}^\tau} n_y y \right)_{\tau \in T_L} \right)_{L \in S} \in \coprod_{\substack{S \subset \frakF_k\\ \textup{$S\neq \emptyset$ finite}} }  \coprod_{\substack{(T_L \subset H^1(L,G))_{L \in S} \\ \textup{$T_L\neq \emptyset$ finite for all $L \in S$}}}  \prod_{L \in S} \prod_{\tau \in T_L} Z_0(Y^\tau_L).\]
Then $\textbf{g}_\ast(z)$ is the $0$-cycle on $X$ given by 
\begin{align*} 
\textbf{g}_\ast(z) & =\sum_{L \in S} \sum_{\tau \in T_L} \sum_{y \in Y_{L}^\tau} n_y (s_{L} \circ g_{L}^\tau)_\ast(y) \\
& = \sum_{L \in S}\sum_{\tau \in T_L} \sum_{y \in Y_{L}^\tau} n_y [ L(y) : k( s_{L} \circ g_{L}^\tau(y) ) ] (s_{L} \circ g_{L}^\tau)(y),
\end{align*}
where  $g_{L}^\tau \colon Y^\tau_{L} \to X_L$ and $s_{L} \colon X_{L} \to X$ are the natural morphisms.
\end{defn}

The map $\textbf{g}_\ast$ is compatible with degrees of 0-cycles in the following way.

\begin{lemma} \label{lem: degrees} Fix an integer $d$. Then we have a map
\[ \mathbf{g}_\ast \colon \coprod_{S \subset \frakF_k}  \coprod_{(T_L \subset H^1(L,G))_{L \in S}}   \coprod_{\substack{\left((\Delta_\tau)_{\tau \in T_L}\right)_{L \in S}: \\ \Delta_\tau \in \Z \textrm{ for all $\tau$}, \\ \sum_{L \in S} \sum_{\tau \in T_L} \Delta_\tau [L:k]= d}} \prod_{L \in S} \prod_{\tau \in T_L} Z_0^{\Delta_\tau}(Y^\tau_L)  \to Z_0^d(X).\]
\end{lemma} 
\begin{proof} 
Let 
\[z := \left(\left(\sum_{y \in Y_{L}^{\tau}} n_y y\right)_{\tau \in T_L}\right)_{L\in S} \in \coprod_{S \subset \frakF_k}  \coprod_{(T_L \subset H^1(L,G))_{L \in S}}   \coprod_{\substack{\left((\Delta_\tau)_{\tau \in T_L}\right)_{L \in S}: \\ \Delta_\tau \in \Z \textrm{ for all $\tau$}, \\ \sum_{L \in S} \sum_{\tau \in T_L} \Delta_\tau [L:k]= d}} \prod_{L \in S} \prod_{\tau \in T_L} Z_0^{\Delta_\tau}(Y^\tau_L).\] 
Then, by definition of $\textbf{g}_\ast$, we have that
\[ \textbf{g}_\ast(z) = \sum_{L \in S}\sum_{\tau \in T_L} \sum_{y \in Y^\tau_{L}} n_y [L(y): k(s_{L} \circ g_{L}^{\tau}(y))]  (s_{L} (g_{L}^{\tau}(y))).\]

Therefore,
\begin{align*}
\deg(\textbf{g}_\ast(z)) &= \sum_{L \in S}\sum_{\tau \in T_L}  \sum_{y \in Y^\tau_{L}} n_y [L(y): k(s_{L} (g_{L}^{\tau}(y)))][k(s_{L} (g_{L}^{\tau}(y))):k] \\
 &= \sum_{L \in S}\sum_{\tau \in T_L}  \sum_{y \in Y^\tau_{L}} n_y [L(y): k] \\
& = \sum_{L \in S}\sum_{\tau \in T_L} \sum_{y \in Y^\tau_{L}}  n_y [L(y): L] [L : k] \\ 
&= \sum_{L \in S}\sum_{\tau \in T_L} [L : k]  \sum_{y \in Y^\tau_{L}} n_y [L(y): L] \\
& = \sum_{L \in S}\sum_{\tau \in T_L} [L : k] \Delta_\tau\\
&= d,
\end{align*}
where the fifth equality follows from the fact that the inner summation represents the degree of the 0-cycle $\sum_{y \in Y^{\tau}_{L}} n_y y$, which is by definition $\Delta_\tau$, and the final equality is by definition of the integers $\Delta_\tau$.
\end{proof} 

\begin{defn} \label{def} Let $d \in \Z$. The \defi{global descent set of degree $d$ of $X$ associated to $g$} is
\[ Z_0^d(X)^g := \bigcup_{\substack{S \subset \frakF_k\\ \textup{$S\neq \emptyset$ finite}} } \textbf{g}_\ast \left(  
  \coprod_{\substack{(T_L \subset H^1(L,G))_{L \in S} \\ \textup{$T_L\neq \emptyset$ finite for all $L \in S$}}} 
 \coprod_{\substack{\left((\Delta_\tau)_{\tau \in T_L}\right)_{L \in S}: \\ \Delta_\tau \in \Z \textrm{ for all $\tau$}, \\ \sum_{L \in S} \sum_{\tau \in T_L} \Delta_\tau [L:k]= d}} \prod_{L \in S} \prod_{\tau \in T_L} Z_0^{\Delta_\tau}(Y^\tau_L)  \right).\]
\end{defn}

The following is the 0-cycles analogue of the partition \eqref{partition}.

\begin{prop} \label{prop:partition} There is an equality of sets, $Z_0^d(X)^g = Z_0^d(X)$.
\end{prop}

\begin{proof} The forward containment is the content of Lemma \ref{lem: degrees}.
For the reverse inclusion, let $z := \sum \limits_{x \in X} n_x x \in Z_0^d(X)$, and suppose that 
$\{x_1, \dots, x_r \} := \supp(z).$
By the classical properties of torsors for rational points (see \cite{Sko01}*{\S 5.3}), for each $x_i$ we take any closed point $x_i'$ of $X_{k(x_i)}$ projecting to $x_i$, and we have the partition
\[ x_i' \in X(k(x_i)) = \coprod_{\sigma \in H^1(k(x_i), G)} g^\sigma_{k(x_i)}\left( Y^\sigma_{k(x_i)}(k(x_i))\right).\]

Hence there exists some (unique) $\sigma_i \in H^1(k(x_i),G)$ such that $x_i$ lifts to some $y_i \in Y^{\sigma_i}_{k(x_i)}(k(x_i))$, which we fix for each $x_i$ and which we can consider as closed points on $Y^{\sigma_i}_{k(x_i)}$. Note that $k(x_i)(y_i) = k(x_i)$, since $ k(x_i)(y_i) \subset k(x_i)$ as $y_i$ is a $k(x_i)$-point.
So, if we consider the morphism
\[ s_{k(x_i)} \circ g_{k(x_i)}^{\sigma_i} \colon Y_{k(x_i)}^{\sigma_i} \to X_{k(x_i)} \to X, \]
then the pushforward of $y_i \in Z_0(Y^{\sigma_i}_{k(x_i)})$ is
\begin{align*}(s_{k(x_i)} \circ g_{k(x_i)}^{\sigma_i})_{\ast}(y_i) &= [k(x_i)(y_i) : k(s_{k(x_i)} \circ g_{k(x_i)}^{\sigma_i}(y_i)) ] x_i \\
& = [k(x_i)(y_i):k(x_i)] x_i \\
& = x_i \in Z_0(X). 
\end{align*}

Consider the set \[M :=\{(k(x_i), \sigma_i) : i = 1, ..., r\}  = \{ (L_1, \tau_1), \dots, (L_s, \tau_s)\},\]
where $L_i \in \{k(x_1), \dots, k(x_r)\}$ and $\tau_i \in \{\sigma_1, ..., \sigma_r\}$.
We partition the set $\{x_1, \dots, x_r\}$ as follows. The points $x_i$ and $x_j$ with $i \ne j$ belong to the same partition set $P_{(L, \tau)}$ if and only if $L = k(x_i) = k(x_j)$ and $\tau = \sigma_i = \sigma_j$. We write 
\[ \{x_1, \dots, x_r\} = \bigsqcup_{i=1}^s P_{(L_i, \tau_i)}.\]
For each $(L_i, \tau_i) \in M$, define
\[ \Delta_{(L_i, \tau_i)} := \sum_{x \in P_{(L_i, \tau_i)}} n_x [L_i(y) : L_i] = \sum_{x \in P_{(L_i, \tau_i)}} n_x,\]
where we recall that $y$ denotes the fixed lift of $x$ and the $n_x$ are the coefficients of the support of the $0$-cycle $z \in Z_0^d(X)$. Then, for each $(L_i, \tau_i)$, we have
\[ \sum_{x \in P_{(L_i, \tau_i)}} n_x y \in Z_0^{\Delta_{(L_i, \tau_i)}}(Y_{L_i}^{\tau_i}). \]
Let $S:= \{L_1, ..., L_s\}$. For each $L \in S$, we let $T_{L} := \{ \tau : (L, \tau) \in M\}$.
Consider the tuple of $0$-cycles 
\[ w := \left( \sum \limits_{x \in P_{(L_i, \tau_i)}} n_x y \right)_{ i \in \{1, \dots, s\}} = \left(\left( \sum \limits_{x \in P_{(L, \tau)}} n_x y \right)_{\tau \in T_L}\right)_{ L \in S}.\]
We claim that  $\textbf{g}_\ast(w) = z$. Indeed, we have
\[ \textbf{g}_\ast(w) = \sum_{i = 1}^s \sum_{x \in T_{(L_i,\tau_i)}} n_x (s_{L_i} \circ g_{L_i}^{\tau_i})_\ast(y) \, = \sum_{x \in \supp(z)} n_x x = z. \]
Moreover, we claim that
$w$ is in 
\[ 
\coprod_{\substack{S \subset \frakF_k\\ \textup{$S\neq \emptyset$ finite}} }  \coprod_{\substack{(T_L \subset H^1(L,G))_{L \in S} \\ \textup{$T_L\neq \emptyset$ finite for all $L \in S$}}} 
\coprod_{\substack{\left((\Delta_\tau)_{\tau \in T_L}\right)_{L \in S} \\ \sum_{L \in S} \sum_{\tau \in T_L} \Delta_\tau [L:k]= d}} \prod_{L \in S} \prod_{\tau \in T_L} Z_0^{\Delta_\tau}(Y^\tau_L)  \]
To show this, we check that the condition on the degrees $\Delta_{(L_i, \tau_i)}$ holds. We have 
\begin{align*}
\sum_{L \in S} \sum_{\tau \in T_L} \Delta_{(L, \tau)} [L : k ] &  = \sum_{i=1}^s \Delta_{(L_i, \tau_i)} [L_i : k ] \\
&= \sum_{i=1}^s \sum_{x \in P_{(L_i,\tau_i)}} n_x [L_i:k ]\\
&= \sum_{j=1}^r  n_{x_j} [k(x_j) : k] \\
&= \deg z\\
&= d
\end{align*}
Hence, the inclusion $Z_0^d(X)^g \supset Z_0^d(X)$ holds and thus $Z_0^d(X)^g = Z_0^d(X)$, as required.
\end{proof}

\begin{remark} \label{rem: global_eff} The subset of $Z_0^1(X)^g$ given by
\[ Z_0^1(X)_k^{g, \textrm{eff}} := \textbf{g}_\ast \left(  \coprod_{\substack{T_k \subset H^1(k,G)\\ \textup{$T_k\neq \emptyset$ finite}}} \enskip \coprod_{\substack{(\Delta_\tau)_{\tau \in T_k}: \\ \Delta_\tau \in \Z \textrm{ for all $\tau$,}\\  \sum_\tau \Delta_\tau = 1}} \prod_{\tau \in T_k} Z_0^{\Delta_\tau, \textrm{eff}}(Y^{\tau})  \right),\]
where we have taken $S = \{k\}$ in Definition \ref{def},
is equal to $\coprod_{\tau \in H^1(k,G)} g^\tau(Y^\tau(k))$.
Indeed, any  effective 0-cycle in  $ Z_0^{\Delta_\tau, \textrm{eff}}(Y^{\tau})$ which is not identically zero must have degree $\Delta_\tau > 0$. Thus for $\Delta_\tau \leq 0$,  we have $ Z_0^{\Delta_\tau, \textrm{eff}}(Y^{\tau})  = \emptyset$. (Strictly speaking, when $\Delta_\tau = 0$ we need to also remove the identically zero 0-cycle; we will be a bit imprecise and ignore this minor issue.) 
But the only way in which $\sum_{\tau \in T_k} \Delta_\tau = 1$ for $\Delta_\tau >0$ is if $T_k = \{ \tau\}$ for some $\tau \in H^1(k,G)$ and $\Delta_\tau = 1$.
Hence, we have
\[ Z_0^1(X)_k^{g, \textrm{eff}} = \textbf{g}_\ast \left( \coprod_{\tau \in H^1(k,G)} Z_0^{1, \textrm{eff}}(Y^{\tau})  \right). \]
One can check that $Z_0^{1, \textrm{eff}}(Y^{\tau}) = Y^\tau(k)$. The result then follows by the definition of $\textbf{g}_\ast$.
\end{remark}

\section{The descent and étale-Brauer sets for 0-cycles} \label{sec: local}

In this section, we extend the definitions from the previous section to the local setting and we define the descent and the étale-Brauer sets for 0-cycles.

Recall that, for any variety $V$ over a number field $k$ and any integer $d$, the \defi{set of adelic 0-cycles of degree $d$ of $V$} is the subset  $Z_0^d(V_{\A_k})$ of $\prod_{v \in \Omega_k} Z_0^d(V_{k_v})$ of 0-cycles $(z_v)_v$ such that, for all but finitely many $v \in \Omega_k$, we have that $z_v$ extends to a 0-cycle over some model $\calV \to \Spec(\calO_{k_v})$ of $V_{k_v} \to \Spec(k_v)$. If $V$ is proper, then $Z_0^d(V_{\A_k}) = \prod_{v \in \Omega_k} Z_0^d(V_{k_v})$.

Let $X$ be a variety over a number field $k$. For any integer $d$ and any place $v \in \Omega_k$, we briefly recall how the natural map $\res_v : Z_0^d(X) \to Z_0^d(X_{k_v})$ is defined. 
For $z := \sum_{x \in X} n_x x \in Z_0^d(X)$,  
we let
\[ \res_v(z) := \sum \limits_{x \in X} \sum \limits_{w \in \Omega_{k(x)} : w \mid v}n_x (x)_w,\]
where $(x)_w \in X(k(x)_w)$ is the image of $x$ under the natural inclusion $X(k(x)) \to X(k(x)_w)$. 
For any $x \in X$, 
we have by e.g. \cite{Neu99}*{Ch.\ II, Cor.\ 8.4} that
\[ \sum_{w \in \Omega_{k(x)} : w \mid v} [k(x)_w : k_v] = [k(x) : k] \] 
and thus
\[ \deg(z_v) = \sum_{x \in X} \sum_{w \in \Omega_{k(x)}: w \mid v} n_x [k(x)_w : k_v]  = \sum_{x \in X} n_x [k(x) :k] = \deg(z). \]

\begin{defn} Let $X$ be a variety over a number field $k$. Let $g: Y \to X$ be a $G$-torsor over $X$, where $G$ is a linear algebraic group over $k$.
We define the map \[ \textbf{g}_{\ast, \textrm{ad}} \colon  \coprod_{\substack{S \subset \frakF_k\\ \textup{$S\neq \emptyset$ finite}} }  \coprod_{\substack{(T_L \subset H^1(L,G))_{L \in S} \\ \textup{$T_L\neq \emptyset$ finite for all $L \in S$}}}  \prod_{L \in S} \prod_{\tau \in T_L} Z_0(Y^\tau_{\A_L}) \to Z_0(X_{\A_k})\]
as follows.
We first observe that
\[ Z_0(Y^\tau_{\Adeles_{L}}) \subset \prod_{w \in \Omega_{L}} Z_0(Y^\tau_{L_w}) = \prod_{v \in \Omega_k} \prod_{w \in \Omega_L: w \mid v} Z_0(Y^\tau_{L_w}). \]
Hence, an element $\tilde{y}$ in the domain of $\textbf{g}_{\ast, \textrm{ad}}$  can be written as 
\[ 
\begin{array}{ll}
\tilde{y} &:= \left( \left(\left( \left( \sum_{y \in Y_{L_w}^{\tau}} n_{y}\; y \right)_{w \in \Omega_{L} : w|v}\right)_{v \in \Omega_k} \right)_{\tau \in T_L} \right)_{L \in S}\\
&= \left(\left( \left( \left( \sum_{y \in Y_{L_w}^{\tau}} n_{y}\; y \right)_{w \in \Omega_{L} : w|v}\right)_{\tau \in T_L} \right)_{L \in S} \right)_{v \in \Omega_k},
\end{array}\] 
and we define
\[ \textbf{g}_{\ast, \textrm{ad}}(\tilde{y}) :=  \left( \sum_{L \in S}\sum_{\tau \in T_L} \sum_{w \in \Omega_L: w \mid v} \sum_{y \in Y_{L_w}^{\tau}} n_{y} (s_{L_w} \circ g_{L_w}^{\tau})_\ast (y)\right)_{v \in \Omega_k} \]
where $g_{L_w}^{\tau} \colon Y_{L_w}^{\tau} \to X_{L_w}$ and $s_{L_w} \colon X_{L_w} \to X_{k_v}$ are the natural maps.
\end{defn}

\begin{lemma} \label{lem:addegrees}Let $X$ be a variety over a number field $k$. Let $g: Y \to X$ be a torsor under some linear algebraic group $G$ over $k$.
Fix $d \in \Z$. Then, 
\[ \mathbf{g}_{\ast, \textnormal{ad}} \colon \coprod_{S \subset \frakF_k }  \coprod_{(T_L \subset H^1(L,G))_{L \in S}}   \coprod_{\substack{\left((\Delta_\tau)_{\tau \in T_L}\right)_{L \in S}: \\ \Delta_\tau \in \Z \textrm{ for all $\tau$}, \\ \sum_{L \in S} \sum_{\tau \in T_L} \Delta_\tau [L:k]= d}} \prod_{L \in S} \prod_{\tau \in T_L} Z_0^{\Delta_\tau}(Y_{\A_{L}}^{\tau})  \to Z_0^d(X_{\A_k}).\]
\end{lemma}

\begin{proof} The proof follows in a similar manner to that of Lemma \ref{lem: degrees}. Let 
\[\tilde{y}:= \left( \left(\left( \left( \sum_{y \in Y_{L_w}^{\tau}} n_{y}\; y \right)_{w \in \Omega_{L} : w|v}\right)_{v \in \Omega_k} \right)_{\tau \in T_L} \right)_{L \in S} \]
be in the domain of $\textbf{g}_{\ast, \textrm{ad}}$.
Then, for each $v \in \Omega_k$, we have
\begin{align*} 
\deg(\textbf{g}_{\ast, \textrm{ad}}(\tilde{y})_v) & = \sum_{L \in S} \sum_{\tau \in T_L} \sum_{w \in \Omega_L:w \mid v} \sum_{y \in Y^{\tau}_{L_w}} n_y [L_w(y) : k_v(s_{L_w}(g^\tau_{L_w}(y)))] [k_v(s_{L_w}(g^\tau_{L_w}(y))): k_v] \\
& = \sum_{L \in S} \sum_{\tau \in T_L} \sum_{w \in \Omega_L:w \mid v} \sum_{y \in Y^{\tau}_{L_w}} n_y [L_w(y) : k_v] \\
& = \sum_{L \in S} \sum_{\tau \in T_L} \sum_{w \in \Omega_L: w \mid v} \sum_{y \in Y^{\tau}_{L_w}} n_y [L_w(y) : L_w][L_w : k_v] \\
& = \sum_{L \in S} \sum_{\tau \in T_L} \sum_{w \in \Omega_L: w \mid v} [L_w : k_v] \sum_{y \in Y^{\tau}_{L_w}} n_y [L_w(y) : L_w] \\
& = \sum_{L \in S} \sum_{\tau \in T_L} \Delta_{\tau} \sum_{w \in \Omega_L: w \mid v} [L_w : k_v] \\
& = \sum_{L \in S} \sum_{\tau \in T_L} \Delta_{\tau} [L:k]\\
&= d, 
\end{align*}
where the fifth equality follows from the fact that the inner summation represents the degree of the 0-cycle $\sum_{y \in Y^{\tau}_{L}} n_y y$, which is by definition $\Delta_\tau$, the sixth equality comes from the fact that $\sum_{w \in \Omega_L: w \mid v} [L_w : k_v] = [L:k]$, and the final equality comes from the definition of the integers $\Delta_\tau$. 
\end{proof}

\begin{defn} \label{def1} Let $X$ be a  variety over a number field $k$. Let $g: Y \to X$ be a $G$-torsor over $X$, where $G$ is a linear algebraic group over $k$. Let $d \in \Z$. 
 For any given finite non-empty set $S \subset \frakF_k$, we define the \defi{$g$-descent set of degree $d$ of $X$ associated to $S$} by
\[ Z_0^d(X_{\A_k})^g_S :=
 \textbf{g}_{\ast, \textrm{ad}} \left(  \coprod_{\substack{(T_L \subset H^1(L,G))_{L \in S} \\ \textup{$T_L\neq \emptyset$ finite for all $L \in S$}}}   \coprod_{\substack{\left((\Delta_\tau)_{\tau \in T_L}\right)_{L \in S}: \\ \Delta_\tau \in \Z \textrm{ for all $\tau$}, \\ \sum_{L \in S} \sum_{\tau \in T_L} \Delta_\tau [L:k]= d}} \prod_{L \in S} \prod_{\tau \in T_L} Z_0^{\Delta_\tau}(Y_{\A_{L}}^{\tau})  \right).\]
 
 We define the \defi{$g$-descent set of degree $d$ of $X$} by
\[ Z_0^d(X_{\A_k})^g := \bigcup_{\substack{S \subset \frakF_k\\ \textup{$S\neq \emptyset$ finite}} }  Z_0^d(X_{\A_k})^g_S.\]

More generally, if $\calG$ is a set of $k$-isomorphism classes of linear algebraic groups over $k$, then the \defi{$\calG$-descent set of degree $d$ of $X$} is defined as
\[  Z_0^d(X_{\A_k})^{\calG} := \bigcup_{\substack{S \subset \frakF_k\\ \textup{$S\neq \emptyset$ finite}} }  \bigcap_{G \in \calG} \bigcap_{[g: Y \to X] \in H^1(X,G)} Z_0^d(X_{\A_k})^g_S \]
\end{defn}

\begin{remark} In the above definition, the reason why we can take the union $\bigcup_{\substack{S \subset \frakF_k\\ \textup{$S\neq \emptyset$ finite}}}$ first, before taking the intersection over the linear algebraic groups and torsors, and still have that $Z_0^d(X) \subset Z_0^d(X_{\A_k})^{\calG}$ is the following: since the motivation behind the sets $S$ comes from the residue fields of the points in $\supp(z)$ for all $z \in Z_0^d(X)$, then, as the proof of Proposition \ref{prop:partition} shows, for any $z \in Z_0^d(X)$ there exists some $S$ (depending on $\supp(z)$ only) such that, for any linear algebraic group $G$ over $k$ and any $G$-torsor $g: Y \to X$, we have 
\[z \in \textbf{g}_\ast \left(  
  \coprod_{\substack{(T_L \subset H^1(L,G))_{L \in S} \\ \textup{$T_L\neq \emptyset$ finite for all $L \in S$}}} 
 \coprod_{\substack{\left((\Delta_\tau)_{\tau \in T_L}\right)_{L \in S}: \\ \Delta_\tau \in \Z \textrm{ for all $\tau$}, \\ \sum_{L \in S} \sum_{\tau \in T_L} \Delta_\tau [L:k]= d}} \prod_{L \in S} \prod_{\tau \in T_L} Z_0^{\Delta_\tau}(Y^\tau_L)  \right),\]
 that is, the same $S$ works for all linear algebraic groups and torsors. We note that the set $ Z_0^d(X_{\A_k})^{\calG} $, as defined above, is potentially smaller than the set
 \[ \bigcap_{G \in \calG} \bigcap_{[g: Y \to X] \in H^1(X,G)} \bigcup_{\substack{S \subset \frakF_k\\ \textup{$S\neq \emptyset$ finite}} }  Z_0^d(X_{\A_k})^g_S. \]
 \end{remark}

\begin{remark} If we set $d=1$, $S=\{k\}$ in Definition \ref{def1}, and we restrict to effective 0-cycles only, then $Z_0^{1, \textrm{eff}}(X_{\AA_k})_k^g = X(\A_k)^g$. Indeed, the proof follows the same argument as that of Remark \ref{rem: global_eff}. 
\end{remark}

\begin{defn} \label{def:fBrSet}
Let $X$ be a  variety over a number field $k$. Let $g: Y \to X$ be a $G$-torsor over $X$, where $G$ is a linear algebraic group over $k$. Let $d \in \Z$. For any given finite non-empty set $S \subset \frakF_k$, we define the \defi{$g$-Brauer set of degree $d$ of $X$ associated to $S$} by
\[ Z_0^d(X_{\A_k})^{g, \Br}_S :=  \textbf{g}_{\ast, \textrm{ad}} \left(  \coprod_{\substack{(T_L \subset H^1(L,G))_{L \in S} \\ \textup{$T_L\neq \emptyset$ finite for all $L \in S$}}} \coprod_{\substack{\left((\Delta_\tau)_{\tau \in T_L}\right)_{L \in S}: \\ \Delta_\tau \in \Z \textrm{ for all $\tau$}, \\ \sum_{L \in S} \sum_{\tau \in T_L} \Delta_\tau [L:k]= d}} \prod_{L \in S} \prod_{\tau \in T_L} Z_0^{\Delta_\tau}(Y_{\A_{L}}^{\tau})^{\Br}   \right).\]
We define the \defi{$g$-Brauer set of degree $d$ of $X$} by
\[ Z_0^d(X_{\A_k})^{g, \Br} := \bigcup_{\substack{S \subset \frakF_k\\ \textup{$S\neq \emptyset$ finite}} } \textbf{g}_{\ast, \textrm{ad}} \left(  \coprod_{\substack{(T_L \subset H^1(L,G))_{L \in S} \\ \textup{$T_L\neq \emptyset$ finite for all $L \in S$}}}  \coprod_{\substack{\left((\Delta_\tau)_{\tau \in T_L}\right)_{L \in S}: \\ \Delta_\tau \in \Z \textrm{ for all $\tau$}, \\ \sum_{L \in S} \sum_{\tau \in T_L} \Delta_\tau [L:k]= d}} \prod_{L \in S} \prod_{\tau \in T_L} Z_0^{\Delta_\tau}(Y_{\A_{L}}^{\tau})^{\Br}   \right).\]

The \defi{étale-Brauer set of degree $d$ of $X$} is defined by
\[ Z_0^d(X_{\A_k})^{ \etBr} := \bigcup_{\substack{S \subset \frakF_k\\ \textup{$S\neq \emptyset$ finite}} }\bigcap_{F \textrm{ finite lin.\ alg.\ $k$-group}} \, \bigcap_{[f: Y \to X] \in H^1(X,F)} Z_0^d(X_{\A_k})^{f, \Br}_S.\]
\end{defn}

The constructions above are all functorial, as the next proposition shows.
\begin{prop} Let $X$ and $Y$ be varieties over $k$ and let $\phi\colon Y \to X$ be a morphism of $k$-varieties.  Let $G$ be a linear algebraic group over $k$ and let $g \colon W \to X$ be a $G$-torsor over $X$. If $\tilde{g} \colon V \to Y$ is the $G$-torsor over $Y$ obtained by pulling $g \colon W \to X$ back along $\phi: Y \to X$, then $\phi$ induces a map of sets
\[ \phi\colon Z_0^d(Y_{\A_k})^{\tilde{g}} \to Z_0^d(X_{\A_k})^g.\]
In particular, we also have an induced map of sets $ \phi\colon Z_0^d(Y_{\A_k})^{\etBr} \to Z_0^d(X_{\A_k})^{\etBr}$.
\end{prop}

\begin{proof}   Let $(\tilde{z}_v)_v := \left(\sum_{y_v \in Y_{k_v}} n_{y_v} y_v\right)_v \in Z_0^d(Y_{\Adeles_k})^{\tilde{g}}$ and let \[ (z_v)_v := \phi_\ast((\tilde{z}_v)_v)  =  \left(\sum_{y_v \in Y_{k_v}} n_{y_v} {\phi_v}_\ast(y_v)\right)_v  \in Z_0^d(X_\Adeles). \]

    We claim that $(z_v)_v \in Z_0^d(X_{\Adeles_k})^{g}$. 
        Since $(\tilde{z}_v)_v \in Z_0^d(Y_\Adeles)^{\tilde{g}}$, there exist a non-empty $S$,  $(T_L)_{L \in S}$ and $\left(\left(\Delta_{\sigma}\right)_{\sigma \in T_L}\right)_{T_L \in S}$ with $\sum_{L \in S} \sum_{\sigma \in T_L} \Delta_{\sigma} [L:k] = d$ and $0$-cycles 
    \[ \left(\left( \left(\sum_{v_w^\sigma \in V_{L_w}^\sigma} \tilde{n}_{v_w}^\sigma v_w^\sigma \right)_{w \in \Omega_L}
    \right)_{\sigma \in T_L}\right)_{L \in S} \in \prod_{L \in S} \prod_{\sigma \in T_L} Z_0^{\Delta_\sigma}(V^{\sigma}_{\Adeles_L}) \] that recombine to $(\tilde{z}_v)_v$, i.e., for all $v \in \Omega_k$, we have 
    \begin{equation}\label{eqn:star}
        \tilde{z}_v = \sum_{L \in S} \sum_{\sigma \in T_L} \sum_{w \in \Omega_L: w \mid v} \, \sum_{v_w^\sigma \in V_{L_w}^\sigma} \tilde{n}_{v_w}^\sigma (\tilde{s}_{L_w}^\sigma \circ \tilde{g}_{L_w}^\sigma )_{\ast}(v_w^\sigma).
    \end{equation}
    Now, since $\tilde{g} \colon V \xrightarrow{G} Y$ arose as the pullback of $g \colon W \xrightarrow{G} X$, we have, for any $L \in S$ and $\sigma \in T_L$, the pullback diagram
    \begin{equation}\label{eqn:pullback} \begin{tikzcd} V_L^\sigma \ar{d}[swap]{\tilde{g}_L^\sigma} \ar{r}{\varphi_L^\sigma} &W_L^\sigma \ar{d}{g_L^\sigma} \\
    Y_L^\sigma \ar{r}{\phi^\sigma_L} & X_L^\sigma
    \end{tikzcd}. \end{equation}
    Moreover, there is a commutative diagram 
    \begin{equation} \label{eqn: basechange} \begin{tikzcd} Y_L \ar{d}[swap]{\tilde{s}_L} \ar{r}{\phi_L} &X_L \ar{d}{s_L} \\
    Y \ar{r}{\phi} & X
    \end{tikzcd}, \end{equation}
    where $\tilde{s}_L$ and $s_L$ are the natural maps.
    Hence, by using the push-forward map 
    \begin{equation} \label{pushf}Z_0^{\Delta_\sigma}(V_{\Adeles_L}^\sigma) \xrightarrow{{\varphi}_{L,\ast}^\sigma} Z_0^{\Delta_\sigma}(W_{\Adeles_L}^\sigma),\end{equation}  
    the 0-cycles $\left(\sum_{v_w^\sigma \in V_{L_w}^\sigma} \tilde{n}_{v_w}^\sigma v_w^\sigma \right)_{w \in \Omega_L}$ push-forward to $0$-cycles 
    \( \left( \sum_{v_w^\sigma \in V_{L_w}^\sigma } \tilde{n}_{v_w}^\sigma {\varphi}_{L,\ast}^\sigma(v_w^\sigma) \right)_{w \in \Omega_L} \in Z_0^{\Delta_\sigma}(W_{\Adeles_L}^\sigma). \)
    In particular, 
    \[ \textbf{g}_{\ast,\textrm{ad}} \left(\left(\left( \left( \sum_{v_w^\sigma \in V_{L_w}^\sigma } \tilde{n}_{v_w}^\sigma {\varphi}_{L,\ast}^\sigma(v_w^\sigma) \right)_{w \in \Omega_L}  \right)_{\tau \in T_L}\right)_{L \in S} \right)\in Z_0^d(X_{\A_k})^g.\]
    Hence, it suffices to show that the 0-cycles  \(\left(\left( \left( \sum_{v_w^\sigma \in V_{L_w}^\sigma } \tilde{n}_{v_w}^\sigma {\varphi}_{L,\ast}^\sigma(v_w^\sigma) \right)_{w \in \Omega_L}  \right)_{\tau \in T_L}\right)_{L \in S}\)
     recombine to $(z_v)_v$. Indeed, for all $v \in \Omega_k$, we have

\begin{align*}
   &\sum_{L \in S} \sum_{\sigma \in T_L} \sum_{w \in \Omega_L: w \mid v} \sum_{v_w^\sigma \in V_{L_w}^\sigma} \tilde{n}_{v_w}^\sigma (s_{L_w}\circ g_{L_w}^\sigma )_{\ast}(\varphi^\sigma_{L_w, \ast}(v_w^\sigma)) \\
   =&  \sum_{L \in S} \sum_{\sigma \in T_L}\sum_{w \in \Omega_L: w \mid v} \sum_{v_w^\sigma \in V_{L_w}^\sigma} \tilde{n}_{v_w}^\sigma (s_{L_w}\circ g_{L_w}^\sigma \circ \varphi_L^\sigma)_{\ast}(v_w^\sigma) \\
    =&  \sum_{L \in S} \sum_{\sigma \in T_L} \sum_{w \in \Omega_L: w \mid v} \sum_{v_w^\sigma \in V_{L_w}^\sigma} \tilde{n}_{v_w}^\sigma (s_{L_w}\circ \phi_{L_w} \circ \tilde{g}_{L_w}^\sigma )_{\ast}(v_w^\sigma) \\
   = &  \sum_{L \in S} \sum_{\sigma \in T_L}\sum_{w \in \Omega_L: w \mid v} \sum_{v_w^\sigma \in V_{L_w}^\sigma} \tilde{n}_{v_w}^\sigma ( \phi_{k_v} \circ \tilde{s}_{L_w} \circ \tilde{g}_{L_w}^\sigma )_{\ast}(v_w^\sigma) \\
   = &  \phi_{k_v,\ast}\left(\sum_{L \in S} \sum_{\sigma \in T_L}\sum_{w \in \Omega_L: w \mid v} \sum_{v_w^\sigma \in V_{L_w}^\sigma} \tilde{n}_{v_w}^\sigma ( \tilde{s}_{L_w} \circ \tilde{g}_{L_w}^\sigma )_{\ast}(v_w^\sigma)\right) \\
    =&  \phi_{k_v,\ast}(\tilde{z}_v)\\
    =&  z_v,
\end{align*}
where the second equality follows from \eqref{eqn:pullback}, the third from \eqref{eqn: basechange}, and the fifth from \eqref{eqn:star}.

For the functoriality of the étale-Brauer set, it suffices to notice that if $(\tilde{z}_v)_v \in Z_0^d(Y_\Adeles)^{\etBr}$, then there exist a non-empty $S$ of extensions of $k$ such that, for any finite linear algebraic group $F$ over $k$ and any $F$-torsor $\tilde{f} \colon V \to Y$, we have
\[ (\tilde{z}_v)_v \in \tilde{\textbf{f}}_{\ast, \textrm{ad}} \left(  \coprod_{\substack{(T_L \subset H^1(L,F))_{L \in S} \\ \textup{$T_L\neq \emptyset$ finite for all $L \in S$}}} \coprod_{\substack{\left((\Delta_\tau)_{\tau \in T_L}\right)_{L \in S}: \\ \Delta_\tau \in \Z \textrm{ for all $\tau$}, \\ \sum_{L \in S} \sum_{\tau \in T_L} \Delta_\tau [L:k]= d}} \prod_{L \in S} \prod_{\tau \in T_L} Z_0^{\Delta_\tau}(V_{\A_{L}}^{\tau})^{\Br}   \right). \]
Let $(z_v)_v :=\phi_\ast((\tilde{z}_v)_v) \in Z_0^d(X_{\A_k})$. We claim that $(z_v)_v \in Z_0^d(X_{\A_k})^{\etBr}$.

If $f \colon W \to X$ is any $F$-torsor and $\tilde{f} \colon V \to Y$ is the pullback of $f$ along $\phi: Y \to X$, then $(\tilde{z}_v)_v \in Z_0^d(Y_{\A_k})^{\tilde{f}, \Br}_S$. Hence, by the functoriality proof above and using the fact that the Brauer-Manin set construction is also functorial, meaning that the push-forward map in \eqref{pushf} induces the map
\[Z_0^{\Delta_\sigma}(V_{\Adeles_L}^\sigma)^{\Br}\xrightarrow{{\varphi}_{L,\ast}^\sigma} Z_0^{\Delta_\sigma}(W_{\Adeles_L}^\sigma)^{\Br},\]
we have that
\[ (z_v)_v  \in Z_0^d(X_{\A_k})^{f,\Br}_S.\]
But since this is true for any finite linear algebraic group $F$ over $k$ and any $F$-torsor $f \colon W \to X$, we have that
\[ (z_v)_v \in \bigcap_{F \textrm{ finite lin.\ alg.\ $k$-group}} \, \bigcap_{[f: Y \to X] \in H^1(X,F)} Z_0^d(X_{\A_k})^{f,\Br}_S \subset Z_0^d(X_{\A_k})^{\etBr}, \]
as required.
\end{proof}

\begin{prop} \label{prop:etBrSubsetBr} Let $X$ be a variety over a number field $k$. Let $g: Y \to X$ be a $G$-torsor over $X$, where $G$ is a linear algebraic group over $k$. Then
$Z_0^d(X_{\A_k})^{g, \Br}  \subset Z_0^d(X_{\A_k})^{\Br} $. In particular, $Z_0^d(X_{\A_k})^{ \etBr}  \subset Z_0^d(X_{\A_k})^{\Br}$.
\end{prop} 
\begin{proof}
Let $x \in Z_0^d(X_{\AA_k})^{g,\Br}$. This means that there exists some
\[ \tilde{y} := \left(\left( \left(\displaystyle \sum_{y \in Y_{L_w}^{\tau}} n_y\; y \right)_{w \in \Omega_{L}} \right)_{\tau \in T_L} \right)_{L \in S} \in \prod_{L \in S} \prod_{\tau \in T_L} Z_0^{\Delta_\tau}(Y_{\A_{L}}^{\tau})^{\Br}  \] 
for some finite $S \neq \emptyset$, some tuple of non-empty finite sets $(T_L \subset H^1(L,G))_{L \in S}$, and some tuple of integers $((\Delta_\tau)_{\tau \in T_L} )_{L \in S}$ with $\sum_{L \in S} [L:k]\sum_{\tau \in T_L} \Delta_\tau = d$,
such that
\[x =  \textbf{g}_{\ast, \textrm{ad}}(\tilde{y}) =  \left( \sum_{L \in S}\sum_{\tau \in T_L} \sum_{w \in \Omega_L: w \mid v} \sum_{y \in Y_{L_w}^{\tau}} n_y (s_{L_w} \circ g_{L_w}^{\tau})_\ast(y)\right)_{v \in \Omega_k}. \]
In particular, we remark that, for any $\tau \in T_L$, 
we have that, for any $\beta \in \Br(Y^\tau_{L})$,
\[  \sum_{w \in \Omega_{L}}  \sum_{y\in Y_{L_w}^{\tau}} n_{y} \, \inv_w \cores_{L_w(y)/L_w} \left(\beta (y)\right) = 0 .\]

We need to show that $x \in Z_0^d(X_{\A_k})^{\Br}$. Let $\gamma \in \Br X$. Fix $L \in S$ and $\tau \in T_L$. Let  $\gamma_v:=\res_{k_v/k}(\gamma) \in \Br(X_{k_v})$, $\alpha:= s_{L_w}^\ast(\gamma_v) \in \Br (X_{L_w})$, and $(g_{L_w}^\tau)^\ast(\alpha) \in \Br (Y^\tau_{L_w})$. Then

\begin{align*}
 0=&  \sum_{w \in \Omega_L}  \sum_{y\in Y_{L_w}^{\tau}} n_{y}  \inv_w \cores_{L_w(y)/L_w} \left((g_{L_w}^\tau)^\ast(\alpha) (y)\right)\\
 =&  \sum_{w\in \Omega_L}  \sum_{y\in Y_{L_w}^{\tau}} n_{y}  \inv_w \cores_{L_w(g_{L_w}^\tau(y))/L_w} \left(\cores_{L_w(y)/L_w(g_{L_w}^\tau(y))} \left((g_{L_w}^\tau)^\ast(\alpha) (y)\right)\right)\\
=&  \sum_{w\in \Omega_L}  \sum_{y\in Y_{L_w}^{\tau}} n_{y} [L_w(y):L_w(g_{L_w}^\tau(y))]  \inv_w \cores_{L_w(g_{L_w}^\tau(y))/L_w} \left(\alpha(g_{L_w}^\tau(y))\right)\\
=&   \sum_{v\in \Omega_k} \sum_{w\in \Omega_L:w|v}   \sum_{y\in Y_{L_w}^{\tau}} n_{y} [L_w(y):L_w(g_{L_w}^\tau(y))]  \inv_w \cores_{L_w(g_{L_w}^\tau(y))/L_w} \left((s_{L_w}^\ast(\gamma_v) )(g_{L_w}^\tau(y))\right)\\
=&  \sum_{v\in \Omega_k} \sum_{w\in \Omega_L:w|v}  \sum_{y\in Y_{L_w}^{\tau}} n_{y} [L_w(y):k_v(s_{L_w}(g_{L_w}^\tau(y)))]  \inv_v \cores_{k_v(s_{L_w}(g_{L_w}^\tau(y)))/k_v} \left(\gamma_v(s_{L_w}(g_{L_w}^\tau(y)))\right),
\end{align*}

where in the third equality we have used restriction-corestriction for the extension $L_w(y)/ L_w(g_{L_w}^\tau(y))$, in the fifth equality we have used the commutative diagram
\[\begin{tikzcd}[column sep = 7 em, row sep = 5 em]
\Br(\Spec(L_w(g_{L_w}^\tau(y)))) \ar{r}{\cores_{L_w(g^\tau_{L_w}(y))/L_w}} \ar{d}{\cores_{L_w(g^\tau_{L_w}(y))/k_v(s_{L_w}(g_{L_w}^\tau(y)))}} &\Br(\Spec(L_w)) \ar{r}{\inv_w} \ar{d}{\cores_{L_w/k_v}} & \Q/\Z \ar{d}{=}\\  
\Br(\Spec(k_v(s_{L_w}(g_{L_w}^\tau(y))))) \ar{r}{\cores_{k_v(s_{L_w}(g_{L_w}^\tau(y)))/k_v}} &  \Br(\Spec(k_v)) \ar{r}{\inv_v} & \Q/\Z 
\end{tikzcd}
\]
together with restriction-corestriction for the extension $L_w(g_{L_w}^\tau(y))/k_v(s_{L_w}(g_{L_w}^\tau(y)))$. 
But then, when considering $ \langle \textbf{g}_{\ast, \textrm{ad}}(y), \gamma \rangle_{BM} $, we get
{\footnotesize
\[\begin{array}{rl}\displaystyle
 \displaystyle \langle \textbf{g}_{\ast, \textrm{ad}}(\tilde{y}), \gamma \rangle_{BM}  & =\displaystyle \sum_{L \in S}\sum_{\tau \in T_L} \sum_{v\in \Omega_k} \sum_{w\in \Omega_L: w|v}  \sum_{y\in Y_{L_w}^{\tau}} n_{y} [L_w(y):k_v(s_{L}(g_{L_w}^\tau(y)))]  \inv_v \cores_{k_v(s_{L}(g_{L_w}^\tau(y)))/k_v} \left(\gamma_v(s_{L}(g_{L_w}^\tau(y)))\right)\\
&= \displaystyle\sum_{L \in S}\sum_{\tau \in T_L}0\\
&= 0,
\end{array}\]
}
that is, $x =\textbf{g}_{\ast, \textrm{ad}}(\tilde{y}) \in Z_0^d(X_{\A_k})^{\Br}$, as required.
\end{proof}

\begin{remark} For rational points, the $g$-descent set $X(\A_k)^g$ can be also written as
\begin{equation}\label{eq111}\bigcup_{\sigma \in H_{\et}^1(k,G)} g^\sigma (Y^\sigma(\A_k)) = \left\{ (x_v)_{v \in \Omega_k} \in X(\A_k) : \exists \ \sigma \in H^1_{\et}(k,G) \textrm{ s.t.} \res_v(\sigma) = [ g_{k_v}^{-1}(x_v)] \   \forall \ v \in \Omega_k \right\}.
\end{equation}
In our construction of $Z_0^d(X_{\A_k})^g$, we have generalised the left-hand side of 
\eqref{eq111}. However, it is not clear, in general, how to generalise the right-hand side of \eqref{eq111} to 0-cycles. Nonetheless, when $G$ is commutative,  we have that $H^1_{\et}(k, G)$ is a \emph{group} and not just a pointed set. Hence, in this case, using the same ideas as in the construction of $Z_0^d(X_{\A_k})^{\Br}$ (i.e. using natural commutative diagrams, corestriction maps, and the fact that $H^1(k_v, G)$ is a group for each $v \in \Omega_k$, implying that we can add its elements together), we can define the set
{\footnotesize
\[ \tilde{Z}_0^d(X_{\A_k})^{g} := \left\{ \left(\sum_{x_v \in X_{k_v}} n_{x_v} x_v \right)_v \in Z_0^d(X_{\A_k}) : \exists \sigma \in H^1(k,G) \textrm{ s.t. } \res_v(\sigma) = \sum_{x_v \in X_{k_v}} n_{x_v} \cores_{k_v(x_v)/k_v}\left( [ g^{-1}_{k_v}(x_v)]\right) \forall \ v \in \Omega_k  \right\},\]}
which is a clear generalisation of the right-hand side of \eqref{eq111}.
It is natural to ask, in this context, what the relationship between $Z_0^d(X_{\A_k})^g$ and $\tilde{Z}_0^d(X_{\A_k})^{g}$ is.
\begin{prop} Let $X$ be a variety over a number field $k$ and let $g: Y \to X$ be a torsor under a commutative linear algebraic group $G$ over $k$.
Let $d \in \Z$. Then $Z_0^d(X_{\A_k})^g \subset \tilde{Z}_0^d(X_{\A_k})^{g}$.
\end{prop}
\begin{proof}
Let   $\left(\sum_{x_v \in X_{k_v}} n_{x_v} x_v \right)_{v \in \Omega_k} \in Z_0^d(X_{\A_k})^g$. Then, by definition, there exist a non-empty finite set $S$ of finite extensions of $k$, a non-empty finite set $T_L \subset H^1_{\et}(L,G)$ for each $L \in S$, a tuple of integers $((\Delta_\tau)_{\tau \in T_L})_{L \in S}$ satisfying $\sum_{L \in S} [L:k] \sum_{\tau \in T_L} \Delta_\tau = d$, and adelic 0-cycles  $\left(\sum_{y_w^\tau \in Y^\tau_{L_w}} n_{y_w^\tau} y_w^\tau\right)_{w \in \Omega_L} \in Z_0^{\Delta_\tau}(Y^\tau_{\A_L})$ with
\begin{equation}\label{eq:equality}
    \sum_{L \in S}\sum_{\tau \in T_L} \sum_{w \in \Omega_L: w \mid v} \sum_{y_w^\tau\in Y_{L_w}^{\tau}} n_{y_w^\tau} (s_{L_w} \circ g_{L_w}^{\tau})_\ast (y_w^\tau)=\sum_{x_v \in X_{k_v}} n_{x_v} x_v,
\end{equation}
for any $v \in \Omega_k$. Let $\alpha := [g: Y \to X] \in H^1_{\et}(X,G)$; for any extension $K/k$, we denote by $\alpha_{K}$ the image of $\alpha$ under the restriction map $\res_{K/k} :H^1_{\et}(X,G) \to H^1_{\et}(X_{K},G)$.

For any $y_w^\tau \in Y^\tau_{L_w}$ such that $s_{k_v}(g_{L_w}^\tau(y_w^\tau)) = x_v$, let $\tilde{y}_w^\tau \in Y_{L_w(y_w^\tau)}$ be an $L_w(y_w^\tau)$-rational point above $y_w^\tau$, let $\tilde{x}_w := g^\tau_{L_w(y_w^\tau)}\in X_{L_w(y_w^\tau)}$, and let $\tilde{x}_v:= s_{L_w(y_w^\tau)/k_v(x_v)}(\tilde{x}_w) \in X_{k_v(x_v)}$, where $s_{L_w(y_w^\tau)/k_v(x_v)}: X_{L_w(y_w^\tau)} \to X_{k_v(x_v)}$ is the natural map. Then $\tilde{x}_w \in X_{L_w(y_w^\tau)}$ is an  $L_w(y_w^\tau)$-rational point above $\tilde{x}_v \in X_{k_v(x_v)}$, which, in turn, is a closed point above $x_v \in X_{k_v}$.

Since
\[ \alpha_{L_w(y_w^\tau)} (\tilde{x}_w) = \res_{L_w(y_w^\tau)/k_v(x_v)}\left(\alpha_{k_v(x_v)}(\tilde{x}_v)\right),\]
it follows, using also restriction-corestriction, that
\[ \begin{array}{ll} \cores_{L_w(y_w^\tau)/k_v}\left(\alpha_{L_w(y_w^\tau)} (\tilde{x}_w) \right) & = [L_w(y_w^\tau): k_v(x_v)] \cores_{k_v(x_v)/k_v} \left(\alpha_{k_v(x_v)}(\tilde{x}_v)\right)\\
& = [L_w(y_w^\tau): k_v(x_v)] \cores_{k_v(x_v)/k_v} \left(\alpha_{k_v}(x_v)\right).\\
\end{array}\]
But we know, by construction, that $ \alpha_{L_w(y_w^\tau)} (\tilde{x}_w) =  \res_{L_w(y_w^\tau)/L}(\tau)$.
Hence,
\begin{equation} \label{eq:rescor}\cores_{L_w(y_w^\tau)/k_v}\left(\res_{L_w(y_w^\tau)/L}(\tau)\right)  = [L_w(y_w^\tau): k_v(x_v)] \cores_{k_v(x_v)/k_v} \left(\alpha_{k_v}(x_v)\right).\end{equation}

Now, for each $x_v \in X_{k_v}$ and for each $L \in S$ and $\tau \in T_L$, we let
\[ Y^\tau_L(x_v):= \{ y_w^\tau \in Y^\tau_L : s_{k_v}(g_{L_w}^\tau(y_w^\tau)) = x_v\}.\]
Then, using \eqref{eq:equality}, we get
\[ \begin{array}{rl} 
\displaystyle
& \displaystyle \sum_{x_v \in X_{k_v}} n_{x_v} \cores_{k_v(x_v)/k_v} \left(\alpha_{k_v}(x_v)\right)\\
= & \displaystyle \sum_{L \in S}\sum_{\tau \in T_L} \sum_{w \in \Omega_L: w \mid v} \sum_{x_v \in X_{k_v}}\sum_{y_w^\tau\in Y^\tau_L(x_v)} n_{y_w^\tau} [L_w(y_w^\tau): k_v(x_v)] \cores_{k_v(x_v)/k_v} \left(\alpha_{k_v}(x_v)\right) \\
 \displaystyle \stackrel{\eqref{eq:rescor}}{=} & \displaystyle\sum_{L \in S}\sum_{\tau \in T_L} \sum_{w \in \Omega_L: w \mid v} \sum_{x_v \in X_{k_v}}\sum_{y_w^\tau\in Y^\tau_L(x_v)} n_{y_w^\tau} \cores_{L_w(y_w^\tau)/k_v}\left(\res_{L_w(y_w^\tau)/L}(\tau)\right) \\
  \displaystyle \stackrel{\textrm{res-cores}}{=} & \displaystyle\sum_{L \in S}\sum_{\tau \in T_L} \sum_{w \in \Omega_L: w \mid v} \sum_{x_v \in X_{k_v}}\sum_{y_w^\tau\in Y^\tau_L(x_v)} n_{y_w^\tau} [L_w(y_w^\tau):L_w] \cores_{L_w/k_v}\left(\res_{L_w/L}(\tau)\right) \\
    \displaystyle = & \displaystyle\sum_{L \in S}\sum_{\tau \in T_L} \sum_{w \in \Omega_L: w \mid v} \cores_{L_w/k_v}\left(\res_{L_w/L}(\tau)\right) \sum_{x_v \in X_{k_v}}\sum_{y_w^\tau\in Y^\tau_L(x_v)} n_{y_w^\tau} [L_w(y_w^\tau):L_w]  \\
    \displaystyle = & \displaystyle\sum_{L \in S}\sum_{\tau \in T_L} \sum_{w \in \Omega_L: w \mid v} \cores_{L_w/k_v}\left(\res_{L_w/L}(\tau)\right) \underbrace{\sum_{y_w^\tau\in Y^\tau_L} n_{y_w^\tau} [L_w(y_w^\tau):L_w]}_{= \Delta_\tau}  \\
    \displaystyle = & \displaystyle\sum_{L \in S}\sum_{\tau \in T_L} \Delta_\tau \sum_{w \in \Omega_L: w \mid v} \cores_{L_w/k_v}\left(\res_{L_w/L}(\tau)\right).  \\
\end{array}\]
From the commutative diagram

\[ \begin{tikzcd} [column sep = 8 em, row sep = 3 em]
H^1_{\et}(L, G) \ar{d}[swap]{\cores_{L/k}} \ar{r}{\bigoplus_{w  \in \Omega_L: w \mid v} \res_{L_w/L}} & \bigoplus_{w  \in \Omega_L: w \mid v}H^1_{\et}(L_w, G) \ar{d}{\sum_{w \mid v} \cores_{L_w/k_v} } \\
   H^1_{\et}(k, G)   \ar{r}{\res_{k_v/k}} & H^1_{\et}(k_v, G),
    \end{tikzcd} \]
we deduce that $\sum_{w \in \Omega_L: w \mid v} \cores_{L_w/k_v}\left(\res_{L_w/L}(\tau)\right) = \res_{k_v/k} \left(\cores_{L/k}\left( \tau\right)\right).$ Hence,

\[ \begin{array}{ll} 
\displaystyle
 \displaystyle \sum_{x_v \in X_{k_v}} n_{x_v} \cores_{k_v(x_v)/k_v} \left(\alpha_{k_v}(x_v)\right)
  & = \displaystyle\sum_{L \in S}\sum_{\tau \in T_L} \Delta_\tau \sum_{w \in \Omega_L: w \mid v} \cores_{L_w/k_v}\left(\res_{L_w/L}(\tau)\right) \\
  &= \displaystyle\sum_{L \in S}\sum_{\tau \in T_L} \Delta_\tau \res_{k_v/k} \left(\cores_{L/k}\left( \tau\right)\right) \\
  &= \res_{k_v/k} \left(\displaystyle\sum_{L \in S}\sum_{\tau \in T_L} \Delta_\tau \cores_{L/k}\left( \tau\right)\right). \\
\end{array}\]
It follows that, if we let
\[ \sigma :=  \sum_{L \in S}\sum_{\tau \in T_L} \Delta_\tau \cores_{L/k}\left( \tau\right) \in H^1_{\et}(k,G),\]
then, for all $v \in \Omega_k$, we have that 
\[\res_{k_v/k}(\sigma) = \sum_{x_v \in X_{k_v}} n_{x_v} \cores_{k_v(x_v)/k_v}\left( \alpha_{k_v}(x_v)\right), \]
that is,  $\left(\sum_{x_v \in X_{k_v}} n_{x_v} x_v \right)_{v \in \Omega_k} \in \tilde{Z}_0^d(X_{\A_k})^g$, as required.
\end{proof}
It is however much less clear whether the other inclusion $\tilde{Z}_0^d(X_{\A_k})^g \subset Z_0^d(X_{\A_k})^{g}$ should hold at all.
\end{remark}

\subsection{Descent obstruction and compatibility with Chow groups and Suslin homology.} When considering the Brauer-Manin set for 0-cycles, if $X$ is a proper variety over a number field $k$, then we know that the Chow group is compatible with the Brauer-Manin pairing, meaning that the Brauer-Manin pairing 
\( \langle - , \rangle_{BM} \colon Z_0^d(X_{\A_k}) \times \Br(X) \to \Q/\Z\) induces  a pairing
\[ \langle - , \rangle_{BM} \colon \prod_{v \in \Omega_k} \CH_0^d(X_{k_v}) \times \Br(X) \to \Q/\Z,\] 
where $\CH_0^d(X_{k_v})$ denotes the set of 0-cycles of degree $d$ on $X_{k_v}$ modulo rational equivalence. It turns out that when $X$ is proper, the descent set also induces an obstruction set modulo rational equivalence. In order to show this, we first need the definition of Suslin's singular homology of degree 0.

For any varieties $V$ and $W$ over a field $K$ with $V$ connected, we define the \emph{group of finite correspondences from $V$ to $W$ over $K$}, to be the group $\textrm{Cor}(V,W)$ whose elements are formal $\Z$-linear sums of integral closed subschemes $Z$ of $W \times_K V$ that are finite and surjective over $V$. For example, if $X$ is a variety over a field $K$, then $\textrm{Corr}(\Spec K, X)$ is just the group of 0-cycles on $X$. 

Consider the points 0 and 1 on the affine line $\A_K^1$. For a variety $X$ over $K$, we define a map \begin{align*} \textrm{Corr}(\A_K^1, X)  &\to \textrm{Corr}(\Spec K, X) \\  \sum_Z n_Z Z &\mapsto \sum_Z n_Z(\iota_0^\ast(Z) - \iota_1^\ast(Z)),
\end{align*}
where, if $\lambda \in \A_K^1$ is a $K$-point, then $\iota_\lambda^\ast(Z)$ denotes the pullback of $Z \subset X \times_K \A^1_K$ along the inclusion $\iota_\lambda: X \to X \times_K \A^1_K$  given by $x \mapsto (x,\lambda)$. Since the composition $Z \subset  X \times_K \A^1_K \to \A_K^1$ is finite and surjective, we can view $\iota_0^\ast(Z) - \iota_1^\ast(Z)$ as a 0-cycle on $X$.
We define the \emph{Suslin's singular homology group of degree 0 on $X$} to be the group $h_0(X)$ of 0-cycles on $X$ modulo the equivalence relation generated by those 0-cycles on $X$ coming from finite correspondences from $\A_K^1$ to $X$ under the map above. Suslin's singular homology of degree 0 behaves well in the following sense (see \cite{ES08}):
\begin{enumerate}[label = (\roman*), itemsep = .25em]
\item if $X$ is a proper variety over a field $K$, then $h_0(X) =  \CH_0(X)$;
\item Let $X$ and $Y$ be varieties over a field $K$. Then, for any $k$-morphism $f\colon Y \to X$, the pushforward map $f_\ast\colon Z_0(Y) \to Z_0(X)$ induces a morphism $f_\ast \colon h_0(Y) \to h_0(X)$;
\item the degree map $\deg \colon Z_0(X) \to \Z$ for 0-cycles factors  over the map $s_\ast: h_0(X) \to h_0(\Spec K) = \Z$ induced by the structure morphism $s\colon X \to \Spec K$.
\end{enumerate}

Even though we will not use it in this paper, we mention here a useful moving lemma for 0-cycles modulo the finite correspondences equivalence, analogous  to \cite[Compl\'{e}ment, \S3]{CT05}.

\begin{prop} $($\cite[Prop. 5.6]{Sch07}$)$ Let $X$ be smooth Noetherian scheme over a field $k$ and let $U$ be a dense open subscheme in $X$. Then the natural homomorphism
\[h_0(U) \to  h_0(X)\]
is surjective. 
\end{prop}

If $d \in \Z$ and $X$ is a variety over a field $K$, we let $h^d_0(X)$ denote the subset of $h_0(X)$ of degree $d$ elements, i.e., $h^d_0(X)$ is the inverse image of $d$ under the degree map $\deg \colon h_0(X) \to \Z$.

\begin{defn} \label{defn:susg} Let $X$ be a variety over a number field $k$. Let $g\colon Y \to X$ be a torsor under a linear algebraic group $G$ over $k$. Let $d \in \Z$.  We define the \emph{$g$-descent set of degree $d$ of $X$ modulo finite correspondences} to be the set
\[ h_0^d(X_{\A_k})^g :=  \textbf{g}_{\ast, \textrm{ad}} \left(  \coprod_{\substack{(T_L \subset H^1(L,G))_{L \in S} \\ \textup{$T_L\neq \emptyset$ finite for all $L \in S$}}} \coprod_{\substack{\left((\Delta_\tau)_{\tau \in T_L}\right)_{L \in S}: \\ \Delta_\tau \in \Z \textrm{ for all $\tau$}, \\ \sum_{L \in S} \sum_{\tau \in T_L} \Delta_\tau [L:k]= d}} \prod_{L \in S} \prod_{\tau \in T_L} \left(\prod_{w \in \Omega_L} h_0^{\Delta_\tau}(Y_{L_w}^{\tau}) \right) \right).\]

\end{defn}

\begin{remark}
If $X$ is proper, then $h_0^d(X_{\A_k})^g = h_0^d(X_{\A_k})^g \cap \prod_{v \in \Omega_k}\CH_0^d(X_{k_v})$.
\end{remark}

\begin{prop} \label{prop: gsus} Let $X$ be a variety over a number field $k$. Let $g: Y \to X$ be a torsor under a linear algebraic group $G$ over $k$. Let $d \in \Z$.  Let $(z_v)_v \in Z_0^d(X_{\A_k})^g$. Let $[(z_v)_v] := \left( [z_v] \right)_v \in \prod_{v} h_0^d(X_{k_v})$ be the class of $(z_v)_v$ under the finite correspondences equivalence. Then $[(z_v)_v]  \in h_0^d(X_{\A_k})^g$.
\end{prop}
\begin{proof} Since $(z_v)_v \in Z_0^d(X_{\A_k})^g$, by definition there exist a non-empty finite set $S$ of finite extensions of $k$, a non-empty finite set $T_L \subset H^1_{\et}(L,G)$ for each $L \in S$, a tuple of integers $((\Delta_\tau)_{\tau \in T_L})_{L \in S}$ satisfying $\sum_{L \in S} [L:k] \sum_{\tau \in T_L} \Delta_\tau = d$, and adelic 0-cycles  $\left(\sum_{y_w^\tau \in Y^\tau_{L_w}} n_{y_w^\tau} y_w^\tau\right)_{w \in \Omega_L} \in Z_0^{\Delta_\tau}(Y^\tau_{\A_L})$ with
\[
	    \sum_{L \in S}\sum_{\tau \in T_L} \sum_{w \in \Omega_L: w \mid v} \sum_{y_w^\tau\in Y_{L_w}^{\tau}} n_{y_w^\tau} (s_{L_w} \circ g_{L_w}^{\tau})_\ast (y_w^\tau)= z_v,\]
for any $v \in \Omega_k$.

For each $L \in S$, each $\tau \in T_L$, and each $w \in \Omega_L$, consider the class 
\[ \left[ \sum_{y_w^\tau \in Y^\tau_{L_w}} n_{y_w^\tau} y_w^\tau \right] =  \sum_{y_w^\tau \in Y^\tau_{L_w}} n_{y_w^\tau} [y_w^\tau]  \in h_0^{\Delta_\tau}(Y^\tau_{L_w}) \]
of $\sum_{y_w^\tau \in Y^\tau_{L_w}} n_{y_w^\tau} y_w^\tau \in Z_0^{\Delta_\tau}(Y^\tau_{L_w})$ under the finite correspondences equivalence. It follows that 
\[  \textbf{g}_{\ast, \textrm{ad}} \left( \left(\left( \left(\left[\sum_{y_w^\tau \in Y_{L_w}^\tau} n_{y^\tau_w} y_w^\tau \right]\right)_{w \in \Omega_L}
    \right)_{\tau \in T_L}\right)_{L \in S}  \right) =  \sum_{L \in S}\sum_{\tau \in T_L} \sum_{w \in \Omega_L: w \mid v} \sum_{y_w^\tau\in Y_{L_w}^{\tau}} n_{y_w^\tau} (s_{L_w} \circ g_{L_w}^{\tau})_\ast ([y_w^\tau]) \]

Since the push-forward map $(s_{L_w} \circ g_{L_w}^{\tau})_\ast \colon Z_0(Y^\tau_{L_w}) \to Z_0(X_{k_v})$ induces a map
\[ (s_{L_w} \circ g_{L_w}^{\tau})_\ast \colon h_0(Y^\tau_{L_w}) \to h_0(X_{k_v}),\]
we have that \( (s_{L_w} \circ g_{L_w}^{\tau})_\ast ( [y^\tau_w] )= [ (s_{L_w} \circ g_{L_w}^{\tau})_\ast (y^\tau_w)]. \) Hence, for each $v \in \Omega_k$, we have
\[ 
\begin{array}{ll}
[z_v] & = \left[ \sum_{L \in S}\sum_{\tau \in T_L} \sum_{w \in \Omega_L: w \mid v} \sum_{y_w^\tau\in Y_{L_w}^{\tau}} n_{y_w^\tau} (s_{L_w} \circ g_{L_w}^{\tau})_\ast (y_w^\tau) \right]\\
 & =  \sum_{L \in S}\sum_{\tau \in T_L} \sum_{w \in \Omega_L: w \mid v} \sum_{y_w^\tau\in Y_{L_w}^{\tau}} n_{y_w^\tau}\left[ (s_{L_w} \circ g_{L_w}^{\tau})_\ast (y_w^\tau) \right]\\
& =  \sum_{L \in S}\sum_{\tau \in T_L} \sum_{w \in \Omega_L: w \mid v} \sum_{y_w^\tau\in Y_{L_w}^{\tau}} n_{y_w^\tau}(s_{L_w} \circ g_{L_w}^{\tau})_\ast \left([y_w^\tau] \right)\\
& =  \textbf{g}_{\ast, \textrm{ad}} \left( \left(\left( \left(\left[\sum_{y_w^\tau \in Y_{L_w}^\tau} n_{y^\tau_w} y_w^\tau \right]\right)_{w \in \Omega_L}
    \right)_{\tau \in T_L}\right)_{L \in S}  \right).
\end{array}\]
It follows that $[(z_v)_v]  \in h_0^d(X_{\A_k})^g$, as required.
\end{proof}

Suslin's singular homology of degree 0 is also compatible with the unramified Brauer-Manin pairing, as the following proposition shows.

\begin{prop} \label{prop:compBrsus} Let $X$ be an integral variety over a number field $k$. Let $B \subset \Br_{nr}(X)$ be a non-empty subset of the unramified Brauer group of $X$. Let $d \in \Z$. Then the Brauer-Manin pairing for 0-cycles
$ \langle - , - \rangle_{BM}\colon \prod_{v} Z_0^d(X_{k_v}) \times B \to \Q/\Z$
is compatible with the finite correspondences equivalence and thus induces a paring
\[ \langle - , - \rangle_{BM}\colon \prod_{v} h_0^d(X_{k_v}) \times B \to \Q/\Z.\]
\end{prop}

\begin{proof} In order to show that the Brauer-Manin pairing is compatible with the finite correspondences equivalence, it suffices to show that, for any $\alpha \in \Br X$, for any  $v \in \Omega_k$, for any elementary finite correspondence $Z \subset X_{k_v} \times \A^1_{k_v}$, we have that the evaluation of $\alpha$ at the 0-cycle $\iota_0^\ast(Z) - \iota_1^\ast(Z) \in Z_0(X_{k_v})$ is 0, where we recall that $\iota_\lambda: X_{k_v} \to X_{k_v} \times \A^1_{k_v}$ is the inclusion $x \mapsto (x,\lambda)$.

Fix $\alpha \in \Br(X)$, a place $v \in \Omega_k$, and an elementary finite correspondence $Z \subset X_{k_v} \times \A^1_{k_v}$. By definition, the projection of $ X_{k_v} \times \A^1_{k_v}$ onto its second factor induces a finite surjective morphism $f\colon Z \to \A^1_{k_v}$, while projection onto its first factor induces a morphism $p\colon Z \to X_{k_v}$. We follow, with some modifications, the proof of \cite{CTS21}*{Proposition 6.4.2}.

Since $Z$ is integral and $\A^1_{k_v}$ is a normal, locally Noetherian scheme of dimension 1 over $k_v$, and since $f$ is surjective (and thus non-constant), it follows that $f$ is flat  (see e.g. \cite[Chap.4, Cor.3.10]{Liu06}). Moreover, since $\A^1_{k_v}$ is locally Noetherian and $f$ is finite and flat, it follows from \cite{stacks-project}*{Tag 02KB} that $f$ is finite and locally free of constant rank (since the rank is constant on irreducible components).

Let $z_0 \in Z_0(Z)$ and $z_1 \in Z_0(Z)$ be the 0-cycles associated to the finite schemes $\Spec(A_0) = f^{-1}(0)$ and $\Spec(A_1) = f^{-1}(1)$, respectively. Then $\iota_\lambda^\ast(Z) = p_\ast(z_\lambda) \in Z_0(X_{k_v})$, for $\lambda \in \{0,1\}$. Let $\alpha_v := \res_{k_v/k}(\alpha) \in \Br(X_{k_v})$ and let $\beta := p^\ast(\alpha_v) \in \Br(Z)$. Then, by \cite{CTS21}*{Lemma 6.4.1}, we have
\[ \langle \beta, z_\lambda \rangle_{BM} = \cores_{A_\lambda/k_v}(\beta_{A_\lambda}) \in \Br(k_v),\]
for $\lambda \in \{0,1\}$. Since $f$ is finite and locally free of constant rank, by \cite{CTS21}*{Proposition 3.8.1}
we have that 
\[ \cores_{A_\lambda/k_v}(\beta_{A_\lambda}) = \langle \cores_{Z/\A_{k_v}^1}(\beta), \lambda \rangle_{BM},\]
 for $\lambda \in \{0,1\}$. Since $k_v$ is a perfect field, the natural map $\Br(k_v) \to \Br(\A^1_{k_v})$ is an isomorphism by \cite{CTS21}*{Theorem 5.6.1(viii)}. Hence, $\cores_{Z/\A_{k_v}^1}(\beta) \in \Br(\A^1_{k_v}) = \Br(k_v)$ is a constant class, and thus $\langle \beta, z_0 \rangle_{BM}  = \langle \beta, z_1 \rangle_{BM},$
  that is,
  \( \langle \beta, z_0 - z_1 \rangle_{BM}  = 0.  \)
  Finally, since $\beta := p^\ast(\alpha_v)$, from \cite{CTS21}*{(6.2)} it follows that
  \[ 0 = \langle p^\ast(\alpha_v), z_0 - z_1 \rangle_{BM} = \langle \alpha_v, p_\ast(z_0 - z_1) \rangle_{BM} = \langle \alpha, \iota_0^\ast(Z) - \iota_1^\ast(Z)  \rangle_{BM},\]
  as required.  
\end{proof}

In light of Proposition \ref{prop:compBrsus}, we make the following definition.

\begin{defn}
Let $X$ be a variety over a number field $k$. Let $g: Y \to X$ be a torsor under a linear algebraic group $G$ over $k$ with $Y$ geometrically integral. Let $d \in \Z$.  We define the \emph{$g$-unramified-Brauer set of degree $d$ of $X$ modulo finite correspondences} to be the set
\[ h_0^d(X_{\A_k})^{g, \Br_{nr}} :=   \textbf{g}_{\ast, \textrm{ad}} \left(  \coprod_{\substack{(T_L \subset H^1(L,G))_{L \in S} \\ \textup{$T_L\neq \emptyset$ finite for all $L \in S$}}} \coprod_{\substack{\left((\Delta_\tau)_{\tau \in T_L}\right)_{L \in S}: \\ \Delta_\tau \in \Z \textrm{ for all $\tau$}, \\ \sum_{L \in S} \sum_{\tau \in T_L} \Delta_\tau [L:k]= d}} \prod_{L \in S} \prod_{\tau \in T_L} \left(\prod_{w \in \Omega_L} h_0^{\Delta_\tau}(Y_{L_w}^{\tau}) \right)^{\Br_{nr}} \right).\]
\end{defn}

\begin{remark}
If $X$ is proper, then $h_0^d(X_{\A_k})^{g, \Br_{nr}} = h_0^d(X_{\A_k})^{g, \Br_{nr}}  \cap \prod_{v \in \Omega_k}\CH_0^d(X_{k_v})$.
\end{remark}

\begin{prop} Let $X$ be a variety over a number field $k$. Let $g: Y \to X$ be a torsor under a linear algebraic group $G$ over $k$ with $Y$ geometrically integral. Let $d \in \Z$.  Let $(z_v)_v \in Z_0^d(X_{\A_k})^{g, \Br_{nr}}$. Let $[(z_v)_v] := \left( [z_v] \right)_v \in \prod_{v} h_0^d(X_{k_v})$ be the class of $(z_v)_v$ under the finite correspondences equivalence. Then $[(z_v)_v]  \in h_0^d(X_{\A_k})^{g, \Br_{nr}}$.
\end{prop}
\begin{proof} The proof is very similar to that of Proposition \ref{prop: gsus}, once we note that Suslin's singular homology of degree 0 is compatible with the unramified Brauer group.
\end{proof}

\section{Weak approximation for 0-cycles} \label{sec:WA}

Let $Y$ be a smooth, geometrically integral variety over a field $k$.  Recall that, when $Y$ is proper, the \emph{Chow group} $\CH_0(Y)$ of $Y$ is the quotient of $Z_0(Y)$ by the subgroup generated by all $0$-cycles of the form $\phi_\ast(\textrm{div}_C(g))$, for all $\phi: C \to Y$ proper morphisms over $k$ from normal integral $k$-curves $C$ and for all $g \in k(C)^\times$. In other words, $\CH_0(Y)$ is the quotient of $Z_0(Y)$ by the subgroup of $0$-cycles rationally equivalent to zero.

Now let $X$ be a smooth, geometrically integral variety over a number field $k$.
While weak approximation for 0-cycles is usually defined, when $X$ is proper, by using the Chow groups $\CH_0(X_{k_v})$ for each $v \in \Omega_k$, in this paper we will need a slightly more general definition which works for non-proper varieties as well. Instead of the Chow groups $\CH_0(X_{k_v})$, we thus define weak approximation by using the Suslin's singular homology groups $h_0(X_{k_v})$. We note that, when $X$ is proper, our definition of weak approximation coincides with the classical one.

\begin{defn} Let $X$ be a smooth, geometrically integral variety over a number field $k$. For any $d \in \Z$,  we say that $X$ \emph{satisfies weak approximation for 0-cycles of degree $d$} if the following condition holds: for any positive integer $n$ and for any finite subset $S \subset \Omega_k$ of places of $k$, if $(z_v)_v \in Z_0^d(X_{\A_k})$, then there exists a global 0-cycle $z_{n,S} \in Z_0^d(X)$ such that $z_{n,S}$ and $z_v$ have the same image in $h_0(X_{k_v})/n$ for any $v \in S$. We can also refine the notion of weak approximation by replacing the set $Z_0^d(X_{\A_k})$ by some potentially smaller set $Z_0^d(X_{\A_k})^\omega$ still containing $Z_0^d(X)$, where $\omega$ is some  ``obstruction" compatible with Suslin's singular homology.
\end{defn}

We will need the notion of two 0-cycles \emph{being sufficiently close}.

\begin{defn} $($\cite[{D\'{e}finition 1.1(3)}]{Lia12}$)$ Let $k$ be a number field and let $v \in \Omega_k$ be a place of $v$. Let $Y$ be a variety over $k_v$. Given a closed point $y \in Y$, we fix a $k_v$-embedding $k_v(y) \to \overline{k_v}$, so that we can view $y$ as a $k_v(y)$-point of $Y$. Let $y' \in Y$ be a closed point and let $U_{y} \subset Y(k_v(y))$ be an open neighbourhood of $y$ with respect to the $v$-analytic topology. We say that $y'$ is \emph{sufficiently close to $y$ (with respect to $U_{y}$)} if $y'$ has residue field $k_v(y') = k_v(y)$ and if we can
choose a $k_v$-embedding $k_v(y') \to \overline{k_v}$ such that $y'$, when viewed as a $k_v(y)$-point of $Y$, is contained in $U_{y}$. The general definition for two 0-cycles $z$ and $z'$  on $Y$ of the same degree to be \emph{sufficiently close} (with respect to, say, a system of open neighbourhoods of the points in the support of $z$) is obtained by extending $\Z$-linearly the above definition for closed points.
\end{defn}

\begin{remark} \label{rem:suffcloseBr}
Let $X$ be a smooth variety over a number field $k$. Let $\alpha \in X$ and let $(z_v)_{v\in \Omega_k} \in Z_0^d(X_{\A_k})$. By the the fact that the evaluation of $\alpha$ at $z_v$ is locally constant for all $v \in \Omega_k$, we deduce that, for each $v \in \Omega_k$, there is a system of open neighbourhoods of the support of $z_v$ such that, if $z'_v \in Z_0^d(X_{k_v})$ is sufficiently close to $z_v$ with respect to this system of neighbourhoods (or smaller neighbourhoods), then the evaluation of $\alpha$ at $z'_v$ is the same as the evaluation of $\alpha$ at $z_v$ (in fact, it is the same ``point-wise", for each point in the support of $z'_v$).
\end{remark}

We recall that, in characteristic 0, effective 0-cycles $z_v$ of degree $d$ on $X_{k_v}$ are in one-to-one correspondence with $k_v$-points  $[z_v] \in \Sym^d(X_{k_v})(k_v)$ on the symmetric product. Let  $z_v$ be a 0-cycle of degree $d$ on $X_{k_v}$. Then $z_v$ can be written uniquely as $z_v = z_v^{+} - z_v^{-}$, where $z_v^{+}$ and $z_v^{-}$ are effective 0-cycles of degrees, say, $d_{+}$ and $d_{-}$, respectively, with $d = d_{+} - d_{-}$.
Hence, $z_v$ corresponds to a pair of $k_v$-points $([z_v^{+}], [z_v^{-}]) \in \Sym^{d_{+}}(X_{k_v})(k_v) \times \Sym^{d_{-}}(X_{k_v})(k_v)$. 
It turns out that, under the above identifications, if two 0-cycles are sufficiently close, then they also have the same image in $h_0(X_{k_v})/n$. 

\begin{prop}\label{prop:suffclose}  Let $X$ be a smooth variety over a number field $k$ and let $v \in \Omega_k$ be a place of $k$. Let $d \in \Z$ and let $z_v \in \Z_0^d(X_{k_v})$. Let $n \in \Z_{>0}$. 
There is a system of open neighbourhoods of the points in the support of $z_v$ such that, if $z'_v$ is sufficiently close to $z_v$ with respect to this system of neighbourhoods (or smaller neighbourhoods) of $z_v$, then $z_v$ and $z_v'$ have the same image in $h_0(X_{k_v})/n$. 
\end{prop}

\begin{remark} \label{rem: suffclosesusBr} Let $k$ be a number field.
Let $S \subset \Omega_k$ be a non-empty finite set of places of $k$. Let $X$ be a smooth, geometrically integral variety over $k$. Let $B \subset \Br^{nr}(X)/\Br_0(X)$ be a finite non-empty subset and let $\beta_1, ..., \beta_r \in \Br^{nr}(X)$ be a complete set of representatives for $B$. Let $n \in \Z_{>0}$. Let $(z_v)_v \in \Z_0^d(X_{\A_k})$. Then, it follows from Remark \ref{rem:suffcloseBr} and Proposition \ref{prop:suffclose} that, for each $v \in S$, we can find a system of open neighbourhoods of the points in the support of $z_v$ such that, if $z'_v \in Z_0^d(X_{k_v})$ is sufficiently close to $z_v$ with respect to this system of neighbourhoods (or smaller neighbourhoods), then $z_v$ and $z'_v$ have the same image in $h_0(X_{k_v})/n$ and they have the same evaluation at each $\beta_i$, for $i = 1, ..., r$.
\end{remark}

The proof of Proposition \ref{prop:suffclose} follows immediately provided that we have an analogue of Lemme 1.8 of \cite{Wit12} for Suslin's singular homology groups. The authors are extremely grateful to Olivier Wittenberg for providing a proof of the  following propositions, analogous to \cite{Wit12}*{Lemme 1.8}.

\begin{prop}[Wittenberg] \label{prop: Wittenberg} Let $k$ be a number field and let $v \in \Omega_k$ be a place. Let $X$ be a smooth quasi-projective variety over $k_v$. For any $d \in \Z$, and any $n \in \Z_{>0}$, the map
\[ \Sym^d_{X_{k_v}}(k_v) \to h_0(X_{k_v})/n\]
is locally constant.
\end{prop}
\begin{proof}
For notational convenience, we will let $X:= X_{k_v}$ and $k:=k_v$.
We will prove the result by induction on the dimension $\dim(X)$. If $\dim(X) = 0$, the result follows trivially. So assume that $\dim(X)\geq 1$ and that the statement of the theorem holds for any smooth variety over $k$ of dimension strictly less than $\dim(X)$.

Let $\hat{X} \subset \PP_k^N$ be a smooth compactification of $X$ (noting that $\dim X = \dim \hat{X})$ and let $Y:= \hat{X} - X$. 
Let $z$ be an effective 0-cycle of degree $d$ on $X$. Let $Z:= \supp(z) \subset X$ be the support of $z$, viewed as a reduced scheme. By replacing the embedding $\hat{X} \subset \PP^N_k$ with its composition with a sufficiently-high-degree Veronese embedding if necessary \cite{AK79}*{(7)}, we can assume that there exists a linear subspace $L \subset \PP^N_k$ of codimension $\dim(X) -1$ containing $Z$ and such that the scheme $L \cap \hat{X}$ is a smooth curve and the scheme $L\cap Y$ is \'etale over $k$. Let $D \subset L$ be a linear subspace of codimension 1 in $L$ such that  $D \cap Z = \emptyset$ and such that the scheme $D \cap \hat{X}$ is \'etale over $k$. Let $H \subset \PP^N_k$ be a linear subspace with $\dim(H) = \dim(X) -1$ and $H \cap D = \emptyset$. Let $\pi : X' \to \hat{X}$ be the variety obtained by blowing-up $\hat{X}$ along $D \cap \hat{X}$ and let $p: X' \to H$ be the projection morphism with centre $D$ in $\PP^N_k$. Observe that since $L \cap Y$
is \`etale over $k$, 
one can ensure that, when choosing $D$, $D \cap Y$ is empty so that $Y \subset X'$.

The fibres of $p$ are the intersections of $\hat{X}$ with the linear subspaces of $\PP^N_k$ of codimension $\dim(X)-1$ containing $D$. 
Hence, there exists a point $h \in H(k)$ such that $L\cap \hat{X} = p^{-1}(h)$. Since $L \cap \hat{X}$ is a smooth curve, there is an open set $V \subset H$ with $h \in V$ and such that the induced morphism $p: p^{-1}(V) \to V$ is smooth, projective, and of relative dimension 1. By shrinking $V$ if necessary, we can assume that $p: p^{-1}(V) \cap Y \to V$ is an \'etale morphism.

Since $p$ is smooth at the points of $Z$ and since $Z \cap D = \emptyset$, there exists a closed subvariety $F \subset p^{-1}(V)$ with $Z \subset F$ and $F$ \'etale over $V$ at the points of $Z$ (c.f.\ \cite{Gro67}*{p. 193}). By shrinking the open set $V$ if necessary, we can assume that $F$ is \'etale on $V$, so that $F$ is a smooth variety. By Hironaka's theorem, there is a smooth compactification $F \subset \hat{F}$ such that the inclusion $F \subset X'$ extends to a morphism $\xi: \hat{F} \to X'$. 
 The smooth variety $\hat{F} - \xi^{-1}(Y)$ has $\dim(\hat{F} - \xi^{-1}(Y)) < \dim(X)$. Hence, by our inductive hypothesis, 
there is an open set $\mathcal{U}_F \subset \Sym^d_{F - F \cap \xi^{-1}(Y)}(k) \subset \Sym^d_{\hat{F} - \xi^{-1}(Y)}(k)$  containing $z$, viewed as a $k$-point of $\Sym^d(X)$, such that the map $\Sym^d_{\hat{F} - \xi^{-1}(Y)}(k) \to h_0(\hat{F} - \xi^{-1}(Y))/n$ is constant on $\calU_F$. By composing this map with the morphism $h_0(\hat{F} - \xi^{-1}(Y))/n \to h_0(X)/n  $ induced by the restriction of $\xi$ to $\hat{F} - \xi^{-1}(Y)$, we deduce that the map
\begin{equation}\label{eq:locconst}
    \calU_F \to h_0(X)/n
\end{equation}
is also constant.

Let $K/k$ be a finite Galois extension containing all the residue fields of the points in $\supp(z)$ and let $G := \Gal(K/k)$. Since $F$ is étale over $V$, by the inverse function theorem \cite{Ser92}*{Part II, Ch. III, \S9, Theorem 2} there exists an open neighbourhood $\calV \subset V(K)$ of $h$ and, for each $x \in Z(K)$, an open neighbourhood $\calB_x \subset \hat{X}(K)$ of $x$ such that $\calB_x \cap Y(K) = \emptyset$, the sets $\calB_x$ are pairwise disjoint, and the maps $\calB_x \cap F(K) \to \calV$ induced by $p$ are isomorphisms of analytic varieties.

Let $s_x: \calV \to \calB_x \cap F(K)$ denote the inverse isomorphisms. By shrinking $\calV$ and by replacing each $\calB_x$ with the intersection $\bigcap_{\sigma \in G} \sigma^{-1}(\calB_{\sigma(x)})$ if necessary, we can assume that $\sigma(\calB_x) = \calB_{\sigma(x)}$ for all $\sigma \in G$ and $x \in Z(K)$. Hence, $\calV$ and $\calB:= \bigcup_{x \in Z(K)} \calB_x$ are stable under the action of $G$. Moreover, for each $x \in Z(K)$,  $\calB_x$ is stable under the action of the stabiliser $G_x \subset G$ of $x$. Additionally, for each $x \in Z(K)$, since $s_x^{-1}$ is $G_x$-equivariant, it follows that the map $s_x$ is $G_x$-equivariant as well.

Let $\varphi: \calB \to F(K)$ be the union of the maps $s_x \circ p : \calB_x \to F(K)$. Then $\varphi$ is a continuous $G$-equivariant map, and thus induces the continuous and $G$-equivariant map between the symmetric products of these topological spaces
\[ \Sym^d(\varphi): \Sym^d(\calB) \to \Sym^d(F(K)).\]
Since the subspaces of $\Sym^d(X(K))$ and of $\Sym^d(F(K))$ given by the $G$-invariant elements can be identified with $ \Sym^d_{X}(k)$ and $\Sym^d_F(k)$,  respectively, it follows that the map $\Sym^d(\varphi)$ induces a continuous map
\[ \psi: \calU_0 \to \Sym^d_{X}(k),\]
where $\calU_0 \subset \Sym^d_{X}(k)$ is the set of $G$-invariant elements of $\Sym^d(\calB)$. The set $\calU := \psi^{-1}(\calU_F)$ is an open set of  $\Sym^d_X(k)$ containing $z$.

We now claim that, for each $a \in \calU$, the class of $a$ in $h_0(X)/n$ is equal to the class of $z$. Indeed, given $a \in \calU$, we can decompose the $0$-cycle $a-z$ on $X$ as
\[ a- z = (a - \psi(a)) + (\psi(a) - z). \]
Since both $\psi(a)$  and $z$ are in $\calU_F$ and since the map  \eqref{eq:locconst} is constant, the class of $\psi(a)-z$ in $h_0(X)/n$ is trivial.
It remains to show that the class of the 0-cycle $a - \psi(a)$ is also trivial in $h_0(X)/n$. To see this, we note first that $\supp(a- \psi(a)) \cap Y = \emptyset$. Let $N_{k'/k}$ denote the norm of $k'/k$ for an intermediate field $k \subset k' \subset K$. Then $a - \psi(a)$ can be written as a sum of cycles on $X$ of the form $N_{k'/k}(b- s_x(p(b)))$ for some intermediate field $k'$, $x \in Z(K)$, and $b \in \calB_x \cap X(k')$. By Proposition~\ref{prop: Olivier2} below, and by shrinking the $\calB_x$'s if necessary, it follows that the class of $b- s_x(p(b))$ is trivial in $h_0(X_{k'})/n$ for any $b \in \calB_x \cap X(k')$. Thus the class of $a - \psi(a)$ is also trivial in $h_0(X)/n$, as required.
\end{proof}

\begin{prop}[Wittenberg] \label{prop: Olivier2} Let $k$ be a local field. Let $p \colon X \to V$ be a smooth proper morphism of varieties over $k$, with $V$ smooth over $k$. Let $x \in X(k)$, let $n \in \Z_{>0}$ be invertible in $k$, let $Y \subset X$ be a codimension $1$ subvariety with $x \notin Y$, and let $F \subset X$ be a codimension $1$ subvariety with $x \in F$. Assume that both $Y$ and $F$ are \'etale over $V$ and are disjoint, and that the fibers of $X \to V$ are geometrically irreducible curves. 

Since $F$ is \'etale over $V$, the map $F(k) \to V(k)$ is a local isomorphism around $x$.
Let $s$ denote its inverse in a neighbourhood of $p(x)$.

If $b$ denotes a rational point of $X$ close enough to $x$, then the class of $b-s(p(b))$ in the 0-th Suslin homology modulo $n$  of $X \setminus Y$ is independent of $b$. 
\end{prop}

\begin{proof}
First we note that we are free to base change
along any smooth morphism $W \to V$ endowed with a rational point of $W$ lying above $p(x)$. In particular, after base changing along
the projection $F \to V$, we may assume that $F \to V$ admits a section and, by ignoring the other irreducible components of $F$, that $F \to V$ is an
isomorphism, that is, that $F$ is a section of $p$.  Let $s \colon V \to X$ denote
this section (i.e., $F=s(V)$). 

Then, let us base change along $p \colon X\setminus Y \to V$
itself. Let $V' := X \setminus Y$, $X' := X \times_V V'$ and let us write $p' \colon X' \to V'$ for the
projection onto the second factor. Let $x'$ denote the rational point $(x,x)$ of $X'$.
Let $Y' = Y \times_V V'$ and $F' = F \times_V V'$.  Let $s' \colon V' \to X'$ denote the
base change of $s \colon V \to X$ and let $d:V' \to X'$ denote the canonical section of $p'$,
i.e., the diagonal.  Now with $p' \colon X' \to V'$, $Y'$, $F'$, and $x'$, we are in exactly the
same situation as in the statement of the Proposition, except over $V'$ rather than over $V$. This is advantageous since now in order to prove the original statement, it is
enough to prove the following improved claim, where instead of looking at the difference between an arbitrary point $b$ of the total space with the
point $s(p(b))$ that lies in the same fibre of $p$ and in the given section $F$,
we now look at the difference between two points lying in the same fibre
but in two given sections. That is, 

\emph{Claim: The class of $d(v')-s'(v')$ in the 0-th Suslin homology
modulo $n$ of $X' \setminus Y'$ is independent
of $v'$, if $v'$ denotes a rational point of $V'$ close enough to $p'(x')=x$.}

We prove this claim as follows. We have the two sections $s'$ and $d$
of $p'$ that coincide at $p'(x')$.  We can view their images $s'(V')$ and $d(V')$
as divisors in $X'$ that are disjoint from $Y'$.  Hence these divisors have
classes in the group $H^2_{et}(X',Y',\mu_n)$. This notation means \emph{cohomology of $X'$ relative to $Y'$}, concretely this
means $H^2_{et}(X', j_!\mu_n)$ where $j$ denotes the inclusion of
$X' \setminus Y'$ in $X'$ and $j_!$ is the extension by zero functor; this
relative cohomology group naturally fits into a long exact sequence

\begin{equation} \label{eqn: LES Olivier}
 \dots \to H^2_{et}(X',Y',\mu_n) \to H^2_{et}(X',\mu_n) \to H^2_{et}(Y',\mu_n)
    \to H^3_{et}(X',Y',\mu_n) \to \dots 
    \end{equation}

Now we can pull back to the fibres above $p'(x')=x$.  Write $X'_x$ and
$Y'_x$ for the fibres of $X'$ and $Y'$ above this point.  We have the pull-back map
 \begin{equation} \label{eqn: pullbackrelative} H^2_{et}(X',Y',\mu_n) \to H^2_{et}(X'_x, Y'_x, \mu_n) 
 \end{equation}
and the difference between the classes of $s'(V')$ and of $d(V')$ lies in the
kernel of \eqref{eqn: pullbackrelative}, since $s'$ and $d$ coincide at $p'(x')$.  Now we
argue just in the same way as one proves that evaluation of Brauer classes
is locally constant (see \cite{Qpoints}*{Proposition 8.2.9(a)}). Let $R$ be henselisation of the local ring
of $V'$ at $p'(x')$.  The pull-back map \eqref{eqn: pullbackrelative} factors as the composition of the maps

\begin{equation} \label{eqn: pullbackfactor1}  H^2_{et}(X',Y',\mu_n) \to H^2_{et}(X'_R, Y'_R, \mu_n)
\end{equation}

\begin{equation} \label{eqn: pullbackfactor2} H^2_{et}(X'_R, Y'_R, \mu_n) \to H^2_{et}(X'_x, Y'_x, \mu_n).
\end{equation}

The map \eqref{eqn: pullbackfactor2} is an isomorphism. Indeed, we can write the long exact sequence of relative cohomology ~\eqref{eqn: LES Olivier} for the domain and for the target of \eqref{eqn: pullbackfactor2}. There are
pull-back maps at each level, and by proper base change two out of three of these
maps are isomorphisms, thus by the five lemma, all are. Hence our class $[s'(V')-d(V')]$ lies in the kernel of \eqref{eqn: pullbackfactor1}, and moreover it even lies in the kernel of the map

\begin{equation} \label{eqn: kernel} H^2_{et}(X',Y',\mu_n) \to H^2_{et}(X'_W, Y'_W, \mu_n)
\end{equation}
for some étale $W \to V'$ endowed with a rational point $w$ above $p'(x')$.
Finally, since $W \to V'$ is étale, it induces a local isomorphism $W(k) \to V'(k)$ around $w$.  By locally choosing an inverse of this isomorphism, we conclude, as in the proof of \cite{Qpoints}*{Prop.8.2.9(a)}, that the image of
$[s'(V')-d(V')]$ by the pull-back map

\begin{equation} \label{eqn: pullbackmap} H^2_{et}(X',Y',\mu_n) \to H^2_{et}(X'_{v'}, Y'_{v'}, \mu_n)
\end{equation}
vanishes for all rational points $v'$ of $V'$ close enough to $p'(x')$.
In other words $d(v')-s'(v')$ dies in the target of \eqref{eqn: pullbackmap} for all such $v'$.
Since for a smooth open curve such as $X'_{v'} \setminus Y'_{v'}$, the $0$-th
Suslin homology group modulo $n$ injects into the 
the target of the map \eqref{eqn: pullbackmap}, the claim is proved.
\end{proof}

\section{Extending Liang's strategy to torsors} \label{sec:LiangTorsors}

\subsection{Extending Liang's strategy to torsors.}  \label{subsec:LiangTorsors}
In \cite{Lia13}*{Theorem 3.2.1} Liang proves, under certain geometric assumptions on the $k$-variety $X$, that, if the Brauer-Manin obstruction is the only one for weak approximation of  $K$-rational points on $X_K$ for any finite extension $K/k$, then the Brauer-Manin obstruction is the only one for weak approximation of  0-cycles of degree $1$ on $X$. 

Recall that the \emph{unramified Brauer group} $\Br_{nr}(Y) = \Br_{nr}(k(Y)/k)$ of a smooth, geometrically integral variety $Y$ over $k$ is the subgroup of $\Br(k(Y))$ defined by the intersection of the images of the natural maps $\Br(A) \into \Br(k(Y))$ for all discrete valuation rings $A$ with field of fractions $k(Y)$ such that $k \subset A$. The unramified Brauer group is a birational invariant that can be used even when no smooth, projective model of $Y$ is available.  We note that when $Y$ is proper, $\Br_{nr}(Y) = \Br(Y)$. For further details, see e.g., \cite{CTS21}*{Chapter 6}.

\begin{thm} \label{thm: fBr} Let $X$ be a smooth, proper, geometrically integral variety over a number field $k$. Let $f : Y \to X$ be an $F$-torsor for some linear algebraic group $F$ over $k$ and with $Y$ geometrically integral. Let $d$ be any integer. Assume that
\begin{enumerate}[label = (\roman*)]
\item for any finite extension $K/k$ and any $\tau \in H^1(K, F)$, the quotient $\Br_{nr}(Y^\tau_{K})/\Br_0 (Y^\tau_{K})$ is finite, and there exists a finite extension $K'_\tau$ of $K$ so that for all finite extensions $L$ of $K$ linearly disjoint from $K'_\tau$ over $K$, the homomorphism induced by restriction 
\[ \res_{L/K} : \Br_{nr}(Y^\tau_{K})/\Br_0 (Y^\tau_{K}) \to \Br_{nr}(Y^\tau_L)/\Br_0(Y^\tau_L)\]
is surjective;
\item  for any finite extension $L/k$, we have that $X_L(\Adeles_L)^{f_L, \Br_{nr}} \neq \emptyset $ if and only if $ X(L) \neq \emptyset$ (respectively, if $X_L(\Adeles_L)^{f_L, \Br_{nr}} \ne \emptyset$, then weak approximation holds for $X_L$).

\end{enumerate}
Then $f$-descent with unramified Brauer obstruction is the only obstruction to the Hasse principle (respectively, weak approximation) for 0-cycles of degree $d$ on $X$.
\end{thm}

\begin{proof} We give a proof for weak approximation in the case when $f$ is not proper; the proofs for weak approximation in the case when $f$ is proper or for the Hasse principle are similar. Fix a positive integer $n$ and a finite subset $S \subset \Omega_k$. Fix a closed point $\tilde{x} \in X$. 

 Let $(z_v)_{v \in \Omega_k} \in Z_0^{d}(X_{\A_k})^{f,{\Br_{nr}}}$. Then, by definition,  there exist a non-empty finite set $S'$ of finite extensions of $k$, a non-empty finite set $T_K \subset H^1_{\et}(K,F)$ for each $K \in S'$, a tuple of integers $((\Delta_\tau)_{\tau \in T_K})_{K \in S'}$ satisfying $\sum_{K \in S'} [K:k] \sum_{\tau \in T_K} \Delta_\tau = d$, and adelic 0-cycles  $(\tilde{z}^\tau_w)_{w \in \Omega_K} \in Z_0^{\Delta_\tau}(Y^\tau_{\A_K})^{\Br_{nr}}$ with
\[ \textbf{f}_{\ast, \textrm{ad}}((((\tilde{z}^\tau_w)_{w \in \Omega_K})_{\tau \in T_K})_{K \in S'}) = (z_v)_{v \in \Omega_k}.\]
For each $\tau$, we follow the proof of \cite{Lia13}*{Theorem 3.2.1}, with the following modifications. Liang translates arithmetic information on a proper variety $Z$ to arithmetic information on $Z \times \PP^1$ in part via the isomorphism of $\Br(Z \times \PP^1) \xrightarrow{\sim} \Br(Z)$. In our setting, $Y^\tau_K$ need not be proper, so we use the unramified Brauer group instead since it is a stably-birational invariant \cite{CTS21}*{Cor.\ 6.2.10}, hence $\Br_{nr}(Y^\tau_K \times \PP^1_K) \isom \Br_{nr}(Y^\tau_K)$. 

Without loss of generality, we can further enlarge the set $S$ in the following way.  For each $K \in S'$ and $\tau \in T_K$, since we are assuming that the quotient $\Br_{nr}(Y_K^\tau)/\Br_0(Y^\tau_K)$ is finite, we can fix a complete finite set $\frakR_{(K, \tau)} \subset \Br_{nr}(Y_K^\tau)$ of representatives for $\Br_{nr}(Y_K^\tau)/\Br_0(Y^\tau_K)$.  Hence, we can find a finite set $S_{0,K, \tau} \subset \Omega_K$ and an integral model $\calY_K^\tau$ of $Y_K^\tau$ over $\Spec(\calO_{K, S_{0,K, \tau}})$ such that all the representatives in $\frakR_{(K, \tau)}$ actually come from elements of $\Br(\calY^\tau_K)$. Moreover, we recall that $(\tilde{z}^\tau_w)_{w \in \Omega_K} \in Z_0^{\Delta_\tau}(Y^\tau_{\A_K})$, so that, by definition of adelic 0-cycles, there are only finitely many $w \in \Omega_K$ (depending on the chosen integral model for $Y^\tau_K$) for which $\tilde{z}^\tau_w$ is not integral. Hence, by enlarging $S_{0,K, \tau}$ further if necessary, we can assume that all the places $w \in \Omega_K$ for which $\tilde{z}^\tau_w$ is not integral (with respect to the model $\calY^\tau_K$) are contained in $S_{0,K, \tau}$.
We then enlarge $S$ by including all the (finitely many) places $v \in \Omega_k$ below the (finitely many) places $w \in S_{0,K, \tau}$, for each $K \in S'$ and $\tau \in T_K$. 
After enlarging $S$ as above if needed, it follows that, for any $K \in S'$ and any $\tau \in T_K$, if $w \in \Omega_K$ is not above any place of $S$, then the evaluation $\beta(\tilde{z}_w^\tau)$ is actually in $\Br(\calO_{K_w}) = 0$, for each $\beta \in \frakR_{(K, \tau)}$.

Following the proof by Liang (see also \cite[Proposition 3.3.3]{Lia13} and the proof of \cite[Proposition 3.4.1]{Lia13}) and using assumption $(i)$, using the notion of generalised Hilbertian sets (see \cite[Definition 3.3.1]{Lia13}) for each $K \in S'$ and each $\tau \in T_K$ we can then construct a finite extension $K_\tau/K$, linearly disjoint from $K'_\tau$ over $K$, with 
\[[K_\tau : K] \equiv \Delta_\tau \ \ (\bmod \  n \cdot [k(\tilde{x}): k])\]
and an adelic point $(y^\tau_{w'})_{w' \in \Omega_{K_\tau}} \in Y^\tau(\A_{K_\tau})^{\Br_{nr}}$ with $\sum_{w' \in \Omega_{K_\tau}: w'|w}r_{K_{\tau, w'}/K_w} (y^\tau_{w'})$ sufficiently close  to $\tilde{z}^\tau_w$ for all places $w$ above $S$, where $r_{K_{\tau, w'}/K_w}: Y^{\tau}_{K_{\tau,w'}} \to Y^\tau_{K_w}$. In particular, by Proposition~\ref{prop:suffclose}, for all the places $w \in \Omega_K$ above a place in $S$,  the 0-cycles $\sum_{w' \in \Omega_{K_\tau}: w'|w}r_{K_{\tau, w'}/K_w}(y^\tau_{w'})$ and $\tilde{z}^\tau_w$ have the same image in $h^0(Y^\tau_{K_w})/n$, say
\[ \left[\sum_{w' \in \Omega_{K_\tau}: w'|w} r_{K_{\tau, w'}/K_w}(y^\tau_{w'})\right] = [\tilde{z}^\tau_w] + n \mu^\tau_{w}\]
for some $\mu^\tau_w \in h^0(Y^\tau_{K_w})$.


Now, let $x^\tau_{w'} := f^\tau_{K_\tau}(y^\tau_{w'})$. Then $(x^\tau_{w'})_{w'} \in X(\A_{K_\tau})^{f^\tau_{K_\tau}, \Br_{nr}}$. By assumption ($ii$), there exists a $K_\tau$-rational point $\tilde{x}^\tau \in X(K_\tau)$ such that $(\tilde{x}^\tau)_{w'} = x^\tau_{w'}$ for all $w'$ above $S$.

Consider the global 0-cycle $z := \left(\sum_{K} \sum_{\tau}  s_{K_\tau}( \tilde{x}^\tau)  \right) + \lambda n \tilde{x}$ for some $\lambda \in \Z$ such that $ z \in Z_0^d(X)$. Note that such $\lambda$ exists by construction, since 
\[\deg\left( \sum_{K \in S'} \sum_{\tau \in T_K}  s_{K_\tau}( \tilde{x}^\tau)  \right) =  \sum_{K \in S'}\sum_{\tau \in T_K} [K_\tau: k] \equiv  \sum_{K \in S'}\sum_{\tau \in T_K} \Delta_\tau [K: k] \equiv d \ (\bmod \ n \cdot [k(\tilde{x}):k]).\]

We claim that the 0-cycle $z$ has the same image as $z_v$ in $h^0(X_{k_v})/n$ for all $v \in S$.
Indeed, for each $v \in S$, under the natural map $Z_0^d(X) \to Z_0^d(X_{k_v})$ we send 
\[z \mapsto \res_v(z) :=\left(  \sum_{K \in S'} \sum_{\tau \in T_K}  \sum_{w' \in \Omega_{K_\tau}: w'|v} s_{K_{\tau, w'}}( (\tilde{x}^\tau)_{w'} ) \right)  + \lambda n \sum_{w \in \Omega_{k(\tilde{x})}: w|v} (\tilde{x})_w   \in Z_0^d(X_{k_v})\]
and, working in $h^0(X_{k_v})$, we have that

\begin{talign*}
[\res_v(z)] & = \left[\sum_{K} \sum_{\tau}  \sum_{w' \in \Omega_{K_\tau}: w'|v} s_{K_{\tau, w'}}( (\tilde{x}^\tau)_{w'} ) \right] + \lambda n \left[ \sum_{w \in \Omega_{k(\tilde{x})}: w|v} (\tilde{x})_w  \right]\\
& \equiv  \left[\sum_{K} \sum_{\tau}  \sum_{w' \in \Omega_{K_\tau}: w'|v} s_{K_{\tau, w'}}((\tilde{x}^\tau)_{w'})  \right] \\
& \equiv  \left[\sum_{K} \sum_{\tau}  \sum_{w' \in \Omega_{K_\tau}: w'|v} s_{K_{\tau, w'}} (x^\tau_{w'}) \right] 		\\
& \equiv  \left[\sum_{K} \sum_{\tau}  \sum_{w' \in \Omega_{K_\tau}: w'|v} s_{K_{\tau, w'}}(f^\tau_{K_\tau}(y^\tau_{w'})) \right] \\
& \equiv  \left[\sum_{K} \sum_{\tau}  \sum_{w' \in \Omega_{K_\tau}: w'|v}  s_{K_w}(f^\tau_{K}(r_{K_{\tau, w'}/K_w}(y^\tau_{w'})))\right] \\
& \equiv \left[\sum_{K} \sum_{\tau}  \sum_{w \in \Omega_K: w|v} \sum_{w' \in \Omega_{K_\tau}: w'|w}  (s_{K_w} \circ f^\tau_{K})_\ast (r_{K_{\tau, w'}/K_w}(y^\tau_{w'}))\right] \\
& \equiv \sum_{K} \sum_{\tau}  \sum_{w \in \Omega_K: w|v}  (s_{K_w} \circ f^\tau_{K})_\ast \left(\left[\sum_{w' \in \Omega_{K_\tau}: w'|w} r_{K_{\tau, w'}/K_w}(y^\tau_{w'}\right)\right]) \\
& \equiv \sum_{K} \sum_{\tau}  \sum_{w \in \Omega_K: w|v}  (s_{K_w} \circ f^\tau_{K})_\ast ( [\tilde{z}^\tau_w] + n \mu^\tau_{w}  ) \\
& \equiv \sum_{K} \sum_{\tau}  \sum_{w \in \Omega_K: w|v}  (s_{K_w} \circ f^\tau_{K})_\ast ( [\tilde{z}^\tau_w]   ) \\
& \equiv\left[\sum_{K} \sum_{\tau}  \sum_{w \in \Omega_K: w|v}  (s_{K_w} \circ f^\tau_{K})_\ast ( \tilde{z}^\tau_w   ) \right] \\
& \equiv \left[z_v\right] \ \ \  (\bmod \ n  h^0(X_{k_v})), \\
\end{talign*}
where in the fifth equality we have used the commutative diagram
\[\begin{tikzcd}[column sep = 3 em, row sep = 3 em]
 Y^{\tau}_{K_{\tau,w'}} \ar{r}{r_{K_{\tau, w'}/K_w}}   \ar{d}{f^\tau_{K_\tau}}  &  Y^\tau_{K_w}  \ar{d}{f^\tau_{K}} \\  
  X_{K_{\tau,w'}}  \ar{r}{s_{K_{\tau, w'}/K_w}}  \ar{d}{s_{K_ {\tau, w'}}}& X_{K_w}\ar{d}{s_{K_w}}\\ 
  X_{k_v} \ar{r}{=}& X_{k_v}.
\end{tikzcd} 	
\] 
and where in the seventh and tenth equalities we have used the fact that pushforward maps of sets of 0-cycles induce maps of  Suslin homology groups.
\end{proof}

\begin{remark}
In Theorem \ref{thm: fBr}, we consider only one torsor $f: Y \to X$ at a time. It is a much harder problem to deal with multiple torsors simultaneously (as it could happen when trying to compute, for example, the étale-Brauer obstruction), since the combinatorial compatibilities between the degrees of the fields constructed in the proof of Theorem \ref{thm: fBr} become much more difficult to enforce and check -- at least, in full generality. 
\end{remark}

In the case that $X$ has a closed point with degree a power of a prime $p$, we may relax the assumptions of Theorem \ref{thm: fBr}.
\begin{prop}\label{rel} Let $X$ be a smooth, proper, geometrically integral variety over a number field $k$ and suppose there exists a closed point $\tilde{x}$ of $X$ of degree $[k(\tilde{x}) :k] = p^r$, for some prime $p$. Let $f \colon Y \to X$ be an $F$-torsor for a linear algebraic group $F$ over $k$ and $Y$ geometrically integral. Let $d$ be an integer coprime to $p$. Assume that
\begin{enumerate}[label = (\roman*)]
    \item for any finite extension $K/k$ and any $\tau \in H^1(K, F)$, the quotient $\Br_{nr}(Y^\tau_{K})/\Br_0 (Y^\tau_{K})$ is finite, and there exists a finite extension $K'_\tau$ of $K$ so that for all finite extensions $L$ of $K$ linearly disjoint from $K'_\tau$ over $K$ and with $\gcd([L:K], p) = 1$, the homomorphism induced by restriction 
\[ \res_{L/K} : \Br_{nr}(Y^\tau_{K})/\Br_0 (Y^\tau_{K}) \to \Br_{nr}(Y^\tau_L)/\Br_0(Y^\tau_L)\]
is surjective. 
    \item for any finite extension $L/k$ of degree coprime to $p$, we have that $X_L(\Adeles_L)^{f_L, \Br_{nr}} \neq \emptyset $ if and only if $ X(L) \neq \emptyset$. 
\end{enumerate}
Then $Z_0^d(\Adeles_k)^{f, \Br_{nr}} \ne \emptyset \implies Z_0^1(X) \ne \emptyset$.
\end{prop}

\begin{proof} Let $(z_v)_v \in Z_0^d(X_{\A_k})^{f,\Br_{nr}}$.  Then there exist a non-empty finite set $S$ of field extensions of $k$, non-empty finite sets $T_K \subset H^1_{\et}(K,F)$  for each $K \in S$, and a tuple $((\Delta_\tau)_{\tau \in T_K})_{K \in S}$ of integers satisfying
 $\sum_{K \in S} \sum_{\tau \in T_K} [K:k] \Delta_\tau = d$, and 0-cycles $(z^{\tau}_w)_{w \in \Omega_K} \in Z_0^{\Delta_\tau}(Y^\tau_{\A_K})^{\Br_{nr}} $ for all $\tau \in T_K$ and all $K \in S$ such that
 \[\textbf{f}_{\ast, \textrm{ad}}\left(\left( \left((z^{\tau}_w)_{w \in \Omega_K}\right)_{\tau \in T_K}\right)_{K \in S}\right) = (z_v)_v.\]
 
Since $\sum \Delta_\tau [K:k] = d$ and $\gcd(d,p)=1$, there exists some $\tau$ such that $\gcd(\Delta_\tau [K:k], p)=1$ and hence $\gcd(\Delta_\tau, p)=1$. Under assumption (i), by applying a similar strategy as in the proof of \cite{Lia13}*{Theorem 3.2.1} to the corresponding torsor $Y^{\tau}_K$ we obtain an extension $L_\tau/K$ of degree $[L_\tau: K] \equiv \Delta_\tau \bmod{p}$  and adelic point of $(y_w')_w \in Y^\tau_{L_\tau}(\A_{L_\tau})^{\Br_{nr}(Y^\tau_K)}$. Our assumption (i) guarantees that $(y_w')_w \in Y^\tau_{L_\tau}(\A_{L_\tau})^{\Br_{nr}(Y^\tau_{L_\tau})}$ and thus that $f^\tau ((y'_w)_w) \in X_{L_\tau}(\A_{L_\tau})^{f^\tau, \Br_{nr}}$. Since $L_\tau/k$ has degree $[L_\tau: k] = [L_\tau: K][K:k]$ which is coprime to $p$, our assumption $(ii)$ yields a rational point $x \in X(L_\tau)$, which can be viewed as a closed point on $X$ of degree coprime to $p$. Hence, by taking a suitable linear combination of $x$ and $\tilde{x}$, we get a 0-cycle of degree $1$ on $X$.
\end{proof}

\begin{remark}
If in the statement of Proposition \ref{rel} we consider weak approximation instead, it is unlikely that the above strategy of proof extends to this setting.
\end{remark}

\begin{remark}
See Theorem \ref{thm: Kummer} for a result which uses similar ideas as those in the proof of Proposition \ref{rel} in the context of (twisted) Kummer varieties.
\end{remark}

\section{Some applications} \label{sec: applications}

\subsection{Enriques surfaces}
In this section, we study the arithmetic behaviour of 0-cycles on Enriques surfaces using our newly defined obstruction sets and some recent results on K3 surfaces.
There is indeed a well-known relationship between Enriques surfaces and K3 surfaces: any Enriques surface $X$ over a number field $k$ can be realised as the quotient of a K3 surface $Y$ by a fixed-point-free involution (see \cite{Bea96}*{Prop III.17}). In other words, for any Enriques surface $Y$ over $k$, we have a $\Z/2\Z$-torsor $f: X \to Y$ with $X$ a K3 surface over $k$.

Conjecturally, the qualitative arithmetic behaviour of rational points on K3 surfaces is completely determined by the Brauer-Manin obstruction.

\begin{conj}[Skorobogatov, \cite{Sko09}] \label{conj: Skoro} The Brauer-Manin obstruction is the only obstruction to the Hasse principle and weak approximation of rational points on K3 surfaces over number fields. 
\end{conj}

Some recent evidence towards Conjecture \ref{conj: Skoro} includes \cite{CTSSD98}, \cite{SSD05}, \cite{IS15}, \cite{HS16}. 
In \cite{Ier21}, Ieronymou used Liang's strategy to prove, conditionally on Skorobogatov's conjecture, that the Brauer-Manin obstruction also completely determines the qualitative arithmetic of 0-cycles on K3 surfaces.

\begin{thm}[\cite{Ier21}*{Theorem 1.2}] \label{thm: Ieronymou}  Let $Y$ be a K3 surface over a number field $k$ and fix an integer $d$. 
Suppose that Conjecture \ref{conj: Skoro} holds. Then, for any positive integer $n$, if $(z_v)_v \in Z_0^d(Y_{\A_k})^{\Br}$ then there exists a global 0-cycle $z_{n} \in Z_0^d(Y)$ such that $z_{n}$ and $z_v$ have the same image in $\CH_0(Y_{k_v})/n$ for all $v \in \Omega_k$.
\end{thm}

By using Ieronymou's result and exploiting the K3 coverings of Enriques surfaces, we are able to study the arithmetic behaviour of 0-cycles on Enriques surfaces.

\begin{thm} \label{thm: etBrEnriques} Let $X$ be an Enriques surface over a number field $k$ and let $f: Y \to X$ be a K3 covering of $X$, i.e. a $\Z/2\Z$-torsor over $X$ with $Y$ a K3 surface. Assume that Conjecture \ref{conj: Skoro} is true. Let $d \in \Z$. Then, for any positive integer $n$, if $ (z_v)_v \in Z_0^d(X_{\A_k})^{f, \Br}$ then there exists a global 0-cycle $z_{n} \in Z_0^d(X)$ such that $z_{n}$ and $(z_v)_v$ have the same image in $\CH_0(X_{k_v})/n$ for all $v \in \Omega_k$. 
\end{thm}

\begin{proof} Fix a positive integer $n$. If $(z_v) \in Z_0^d(X_\A)^{f, \Br}$, then by definition there exist a non-empty finite set $S$ of field extensions of $k$, non-empty finite sets $T_K \subset H^1_{\et}(K,\Z/2\Z)$ for each $K \in S$, and a tuple $((\Delta_\tau)_{\tau \in T_K})_{K \in S}$ of integers satisfying
 $\sum_{K \in S} \sum_{\tau \in T_K} [K:k] \Delta_\tau = d$, and 0-cycles $(z^{\tau}_w)_{w \in \Omega_K} \in Z_0^{\Delta_\tau}(Y^\tau_{\A_K})^{\Br} $ for all $\tau \in T_K$ and all $K \in S$ such that
 \[\textbf{f}_{\ast, \textrm{ad}}\left(\left( \left((z^{\tau}_w)_{w \in \Omega_K}\right)_{\tau \in T_K}\right)_{K \in S}\right) = (z_v)_v.\]

 Under the assumption that Conjecture \ref{conj: Skoro} is true, from Theorem \ref{thm: Ieronymou} we deduce that  there exist global 0-cycles $z_\tau \in Z_0^{\Delta_\tau}(Y^\tau_K)$ such that $z_\tau$ and $(z^{\tau}_w)_{w \in \Omega_K}$ have the same image in $\CH_0(Y^\tau_{K_w})/n$ for all $w \in \Omega_K$, for all $\tau \in T_K$ and all $K \in S$.
 
 Let $z := \textbf{f}_\ast\left(\left( \left(z_{\tau}\right)_{\tau \in T_K}\right)_{K \in S}\right)$. By Lemma \ref{lem: degrees}, $z \in Z_0^d(X)$. Moreover, it is not hard to see that $z$ and $z_v$ must have the same image in $\CH_0(X_{k_v})/n$ for all $v \in \Omega_k$. Indeed, for all $w \in \Omega_K$, for all $\tau \in T_K$ and $K \in S'$,  working at the level of Chow groups we have 
 \[ [z_w^\tau] = [\res_{w}(z_\tau)] + n \lambda_w^\tau\]
 for some $\lambda_w^\tau \in \CH_0(Y^\tau_{K_w})$, where
 $\res_w(z_\tau)$ is the image of $z_\tau$ under the natural map $Z^{\Delta_\tau}_0(Y^\tau_K) \to Z^{\Delta_\tau}_0(Y^\tau_{K_w})$. Then, using the fact that proper pushforward maps at the level of sets of 0-cycles induce maps of Chow groups, for every $v \in \Omega_k$ we have
 \begin{talign*}
  [z_v] &= \sum_{K \in S} \sum_{\tau \in T_K} \sum_{w \in \Omega_K : w|v} (s_K \circ f_K^\tau)_\ast \left( [z_w^\tau]\right)\\
  & = \sum_{K \in S} \sum_{\tau \in T_K} \sum_{w \in \Omega_K : w|v} (s_K \circ f_K^\tau)_\ast \left( [\res_w(z_\tau)] + n \lambda_w^\tau\right)\\
  & = [\res_v(z)] + \sum_{K \in S} \sum_{\tau \in T_K} \sum_{w \in \Omega_K : w|v} (s_K \circ f_K^\tau)_\ast \left( n \lambda_w^\tau\right)\\
  & =  [\res_v(z)] + n \sum_{K \in S} \sum_{\tau \in T_K} \sum_{w \in \Omega_K : w|v} (s_K \circ f_K^\tau)_\ast \left( \lambda_w^\tau\right),
  \end{talign*}
 where $\res_v(z)$ is the image of $z$ under the natural map $Z^d_0(X) \to Z^d_0(X_v)$. Since 
 \[\sum_{K \in S} \sum_{\tau \in T_K} \sum_{w \in \Omega_K : w|v} (s_K \circ f_K^\tau)_\ast \left( \CH_0(Y^\tau_{K_w})\right) \subset \CH_0(X_{k_v}),\] we are done.
\end{proof}

\subsection{(Twisted) Kummer varieties as torsors}

Let $A$ be an abelian variety over a number field $k$ of dimension $\geq 2$. Let $\sigma := [T \to \Spec k] \in H^1_{\et}(k, A[2])$. Under the natural morphism $H^1_{\et}(k, A[2]) \to H^1_{\et}(k, A)$, we have that $\sigma$ gives rise to a 2-covering $\rho: V \to A$, where $V$ has the structure of a $k$-torsor under $A$ of period dividing 2. The involution $[-1] : A \to A$, fixing $A[2]$, induces an involution $\iota: V \to V$ fixing $T = \rho^{-1}(0_A)$. Let $\tilde{V} \to V$ be the blow-up of $V$ at $T$. Then the involution $\iota$ induces an involution $\tilde{\iota}: \tilde{V} \to \tilde{V}$ fixing the exceptional divisors of the blow-up. The quotient $ \tilde{V}/\tilde{\iota}$ is called the \emph{(twisted) Kummer variety associated to $A$ and $\sigma$}. In what follows, we sometimes omit the references to $A$ and $\sigma$ and just talk about \emph{(twisted) Kummer varieties}.

\begin{lemma}\label{pow2}
Let $Y$ be a (twisted) Kummer variety over a number field $k$. Then $Y$ has a 0-cycle of degree a power of 2.
\end{lemma}

\begin{proof} Since  $Y$ is a (twisted) Kummer variety, it admits a double cover by a smooth, proper variety $Z$ which is birational to a torsor $V$ under an abelian variety of dimension $g$ of period $P(V)$ dividing $2$. If  $I(V)$ denotes the index of $V$, that is, the gcd of the degrees of all closed points on $V$, then the following divisibility relation between the period and index is well-known:
\[ P(V) \ | \ I(V) \  | \ P(V)^{2g}.\]
In particular, $I(V)$ must be a power of 2. Since, for $k$ a number field, the index is a birational invariant of smooth varieties, it follows that $I(V) = I(Z)$. Hence, by definition of the index, $Z$ has a 0-cycle of degree $I(Z)$, a power of 2. By pushing forward this 0-cycle from $Z$ to $Y$, we obtain a 0-cycle on $Y$ of degree a power of 2, as required.
\end{proof}

Given that there is a close relationship between (twisted) Kummer varieties and $k$-torsors under abelian varieties, and since, conditionally on the finiteness of the relevant Tate-Shafarevich group, the (algebraic) Brauer-Manin obstruction is the only one for the existence of rational points on $k$-torsors under abelian varieties (see e.g., \cites{Man71,Cre20}), it is natural to ask the following question (and to possibly expect a positive answer).

\begin{question} \label{q:kum}
Let $X$ be a (twisted) Kummer variety over a number field $k$. Is it true that, for any finite extension $L/k$ of odd degree, $X(\A_L)^{\Br} \neq \emptyset$ implies $X(L) \neq \emptyset$?
\end{question}

For some evidence towards a positive answer to Question \ref{q:kum}, see for example \cite{SSD05}, \cite{HS16}, and \cite{Har19}.

\begin{thm}\label{thm: Kummer} Let $X$ be a smooth, proper, geometrically integral variety over a number field $k$. Let $f: Y \to X$ be a torsor under some linear algebraic group $F$ over $k$, where $Y$ is a (twisted) Kummer variety over $k$. Let $d \in \Z$ be \emph{odd}. Assume that Question \ref{q:kum} has a positive answer. Then
$Z_0^d(X_{\A_k})^{f,\Br\{2\}} \neq \emptyset$ implies $Z_0^{d}(X) \neq \emptyset$.
\end{thm}

\begin{proof}
Let $(z_v)_v\in Z_0^d(X_{\A_k})^{f,\Br\{2\}}$. Then, by definition, 
there exist a non-empty finite set $S$ of field extensions of $k$, non-empty finite sets $T_K \subset H^1_{\et}(K,F)$ for each $K \in S$, and a tuple $((\Delta_\tau)_{\tau \in T_K})_{K \in S}$ of integers satisfying
 $\sum_{K \in S} \sum_{\tau \in T_K} [K:k] \Delta_\tau = d$, and 0-cycles $(z^{\tau}_w)_{w \in \Omega_K} \in Z_0^{\Delta_\tau}(Y^\tau_{\A_K})^{\Br\{2\}} $ for all $\tau \in T_K$ and all $K \in S$ such that
 \[\textbf{f}_{\ast, \textrm{ad}}\left(\left( \left((z^{\tau}_w)_{w \in \Omega_K}\right)_{\tau \in T_K}\right)_{K \in S}\right) = (z_v)_v.\]
 Since $d$ is odd, it follows that there exists some $K \in S$ and some $\tau \in T_K$ such that $\gcd([K:k] \Delta_\tau, 2) = 1$. By the proof of \cite{Lia13}*{Theorem 3.2.1}, we can construct a finite extension $K_\tau /K$ with $[K_\tau: K] \equiv \Delta_\tau \ (\bmod \ 2)$ such that $ Y^\tau_{K_\tau}(\A_{K_\tau})^{\Br(Y^\tau_K)\{2\}}\neq \emptyset$. But, by construction, $[K_\tau: K]$ is odd. Hence, by  \cite{BN21}*{Lemma 7.1} we know that actually $Y^\tau_{K_\tau}(\A_{K_\tau})^{\Br(Y^\tau_K)\{2\}} = Y^\tau_{K_\tau}(\A_{K_\tau})^{\Br(Y^\tau_{K_\tau})\{2\}}$.
 Since $Y^\tau_{K_\tau}$ is a (twisted) Kummer variety,  \cite{CV18}*{Theorem 5.10} then yields that $Y^\tau_{K_\tau}(\A_{K_\tau})^{\Br(Y^\tau_{K_\tau}) }\neq \emptyset$. By assumption, this implies that there exists some rational point $y \in Y^\tau_K(K_\tau)$, which can be viewed as a closed point of degree $[K_\tau:K]$ on $Y^\tau_K$. Since $Y^\tau_{K}$ is a (twisted) Kummer variety, we know by Lemma \ref{pow2} that $Y^\tau_{K}$ has a global 0-cycle of degree a power of 2. Hence, by Bézout's theorem, there exists a global 0-cycle $\tilde{z}_\tau \in Z_0^1(Y^\tau_K)$ and thus a global 0-cycle $z \in Z_0^{[K:k]}(X)$. If $\# S = 1$, then $[K : k] | d$ and we are done. 
 If $\# S > 1$, then, since $\sum_{K' \in S} \sum_{\tau' \in T_{K'}} [K':k] \Delta_{\tau'} = d$, 
 it follows that
 \[\gcd\left([K:k], [K':k] \textrm{ for $K' \in S - \{K\}$}\right) \  | \ d.\] 
 In particular, since $[K:k]$ is odd, it follows that 
 \begin{equation} \label{eq:gcd}\gcd\left([K:k], 2^{t_{K'}}[K':k] \textrm{ for $K' \in S - \{K\}$}\right) \ | \ d\end{equation}
 for any integers $t_{K'} \geq 0$. For any $K' \in S - \{K\}$, fix some $\tau' \in T_{K'}$. Then, by considering the (twisted) Kummer variety $Y^{\tau'}_{K'}$ and, by Lemma \ref{pow2}, a global 0-cycle  $y'$ of degree a power of 2 on $Y^{\tau'}_{K'}$, we get that the pushforward $(f^{\tau'}_{k'})_\ast(y')$ is a global 0-cycle of degree a power of 2 on $X_{K'}$, and thus a global 0-cycle $x'$ of degree $2^{t_{K'}} [K':k]$ on $X$, for some $t_{K'} \geq 0$. Hence, by \eqref{eq:gcd}, we can take an appropriate linear combination of the 0-cycles $x'$ (for each $K' \in S-\{K\}$) and $z$ we obtain a 0-cycle of degree $d$ on $X$, as required.
 \end{proof}

\begin{remark} By using \cite{BN21}, one can also consider the more general case of torsors under arbitrary finite products of (twisted) Kummer varieties, K3 surfaces, and geometrically rationally connected varieties over some number field.
\end{remark}

\begin{remark}
As already mentioned in the introduction, the recent preprint \cite{Ier22} by Ieronymou should remove some of the conditions in Theorem \ref{thm: Kummer}, namely we can get a statement for any $d \in \Z$ and for the full Brauer groups, with a proof similar to that of Theorem \ref{thm: etBrEnriques}. We have preferred to leave the statement of Theorem \ref{thm: Kummer} in its current form as its proof could potentially be used in other contexts where one has only limited information about the arithmetic of rational points (e.g. when one only knows local-to-global principles for rational points with $2$-primary Brauer groups).
\end{remark}

\subsection{Universal torsors and torsors under tori}
Recall that if $X$ is a variety over $k$ and $g: Y \to X$ is a $G$-torsor over $X$ for some linear algebraic group $G$ over $k$ of multiplicative type, then the \emph{type of the torsor $g: Y \to X$} is  the map
\[ \lambda : \widehat{G} \to \Pic \Xbar\]
which associates to any character $\chi \in \widehat{G} = \widehat{G}(\kbar)$ the class of the pushforward $\chi_\ast(\Ybar) \to \Xbar$ in $H^1(\Xbar, \G_m) = \Pic \Xbar$   (see \cite{Sko01}*{Lemma 2.3.1}), 
 where $\widehat{G} := \Hom_{\textrm{$k$-groups}}(G, \G_m) = \widehat{G}(\kbar)$ is the module of characters of $G$. If, moreover, $\Pic \Xbar$ is finitely generated as a $\Z$-module, then we say that $g: Y \to X$ is a \emph{universal torsor} if the type map $\lambda :  \widehat{G} \to \Pic \Xbar$ is an isomorphism of $\Gal(\kbar/k)$-modules. One useful feature of universal torsors is that they are really defined geometrically, implying that if $g: W \to X$ is a universal torsor for $X$ under some linear algebraic group $G$ over $k$, then $g_K : W_K \to X_K$ is also a universal torsor for $X_K$ (under $G_K$) for any finite extension $K/k$.  Universal torsors satisfy many nice properties, in the context of rational points. For example, we have the following.

\begin{thm}[\cite{Sko01}*{Theorem 6.1.2}]\label{univtpt} Let $X$ be a variety over $k$ with $\kbar[X]^\times = \kbar^\times$ and $\Pic \Xbar$ finitely generated as a $\Z$-module. Assume that $g: W \to X$ is a universal torsor for $X$. Then $X(\A_k)^g = X(\A_k)^{\Br_1}$. 
\end{thm}

In this section, we leverage some of our knowledge of universal torsors in the context of rational points to deduce some information in the context of 0-cycles.

\begin{thm} \label{univt}  
Let $X$ be a smooth, proper, geometrically integral variety over $k$ with $\Pic \Xbar$ finitely generated as a $\Z$-module. Suppose that a universal torsor $g: W \to X$ under some group $G$ of multiplicative type over $k$ exists. Then, for any integer $d \in \Z$, for any positive integer $n$, and for any finite subset $S' \subset \Omega_k$ of places of $k$, we have that
\begin{enumerate}
\item if $(z_v)_v \in \Z_0^d(X_{\A_k})^g $, then there exists some $(u_v)_v \in Z_0^d(X_{\A_k})^{\Br_1}$ such that $z_v$ and $u_v$ have the same image in $\CH_0(X_{k_v})/n$ for all $v \in S'$;
\item if, moreover, $\Br_1(X)/\Br_0(X)$ is finite, then $(z_v)_v \in \Z_0^d(X_{\A_k})^{\Br_1}$ implies that there exists some  $(u_v)_v \in \Z_0^d(X_{\A_k})^{g}$ such that $z_v$ and $u_v$ have the same image in $\CH_0(X_{k_v})/n$ for all $v \in S'$.
\end{enumerate}
\end{thm}

\begin{proof} (1)
Let  $(z_v)_v \in \Z_0^d(X_{\A_k})^g$. Then, by definition, 
there exist a non-empty finite set $S$ of field extensions of $k$, non-empty finite sets $T_K \subset H^1_{\et}(K,G)$ for each $K \in S$, and a tuple $((\Delta_\tau)_{\tau \in T_K})_{K \in S}$ of integers satisfying
 $\sum_{K \in S} \sum_{\tau \in T_K} [K:k] \Delta_\tau = d$, and 0-cycles $(z^{\tau}_{v'})_{v' \in \Omega_K} \in Z_0^{\Delta_\tau}(W^\tau_{\A_K}) $ for all $\tau \in T_K$ and all $K \in S$ such that
 \[\textbf{g}_{\ast, \textrm{ad}}\left(\left( \left((z^{\tau}_{v'})_{v' \in \Omega_K}\right)_{\tau \in T_K}\right)_{K \in S}\right) = (z_v)_v.\]

 For each $\tau \in T_K$ (for each $K \in S$), we can fix a closed point $y_{\tau} \in W^{\tau}_{K}$, with $F_{\tau}:= K(y_{\tau})$ and degree say $\delta_{y_{\tau}} := [F_{\tau}: K]$. Then, arguing as in \cite{Lia13}*{Theorem 3.2.1}, we  can construct an extension $L_{\tau} /K$ of degree 
\[[L_{\tau} : K] \equiv \Delta_{\tau} \ (\bmod \  n \delta_{y_{\tau}} ),\]
say $[L_{\tau} : K] = \Delta_{\tau} + \lambda_{\tau}  n \delta_{y_{\tau}}$ for some integer $\lambda_{\tau}$, and an adelic point $(m_{w})_{w \in \Omega_{L_\tau}} \in W^{\tau}(\A_{L_{\tau}})$ with $\sum_{w \in \Omega_{L_\tau}: w| v'} r_{L_{\tau, w}/K_{v'}}(m_{w})$ sufficiently close to $z_{v'}^\tau$ for all the places $v' \in \Omega_K$ above the places in $S'$, where $r_{L_{\tau, w}/K_{v'}}: W^\tau_{L_{\tau}, w} \to W^\tau_{K_{v'}}$ is the natural map. By Proposition \ref{prop:suffclose},  $\sum_{w \in \Omega_{L_\tau}: w| v'} r_{L_{\tau, w}/K_{v'}}(m_{w})$ and $z_{v'}^\tau$ also have the same image in $h_0(W^\tau_{K_{v'}})/n$.

Let $({x}^{(\tau)}_{w})_{w \in \Omega_{L_\tau}}  := (g_{L_\tau}^\tau(m_{w}))_{w \in \Omega_{L_\tau}}$. Then $({x}^{(\tau)}_{w})_{w \in \Omega_{L_\tau}}  \in X(\A_{L_{\tau}})^{g_{L_{\tau}}}$. But $g_{L_{\tau}}: W_{L_{\tau}} \to X_{L_{\tau}}$ is again a universal torsor (under $G_{L_{\tau}}$). Hence, Theorem \ref{univtpt} yields that  $ ({x}^{(\tau)}_{w})_{w \in \Omega_{L_\tau}} \in X(\A_{L_{\tau}})^{\Br_1(X_{L_{\tau}})}$. Similarly, $(\tilde{{x}}^{(\tau)}_{w'})_{w' \in \Omega_{F_{\tau}}} := (g^\tau_{F_{\tau}}((y_\tau)_{w'}))_{w' \in \Omega_{F_{\tau}}}  \in X(\A_{F_{\tau}})^{\Br_1(X_{F_{\tau}})}$.

Consider the adelic 0-cycle on $X$ whose $v$-adic component is given by
\[ u_v := \sum_{K \in S} \sum_{\tau \in T_K}\sum_{v' \in \Omega_{K} : v'|v}   \left( \sum_{w \in \Omega_{L_{\tau}}: w|v'} s_{L_{\tau, w}/k_v} ({x}^{(\tau)}_w) \  - \lambda_{\tau} n \ \sum_{w' \in \Omega_{F_{\tau}}: w'|v'} s_{F_{\tau, w'}/k_v} ( \tilde{{x}}^{(\tau)}_{w'}) \right), \]
where $s_{L_{\tau, w}/k_v} : X_{L_{\tau, w}} \to X_{k_v}$ and $s_{F_{\tau, w'}/k_v} : X_{F_{\tau, w'}} \to X_{k_v}$ are the natural maps, 
with degree
\[ 
\begin{array}{rl}
\deg(u_v) =& \sum_{K \in S} \sum_{\tau \in T_K} \sum_{v' \in \Omega_{K} : v'|v}  \left( \sum_{w \in \Omega_{L_{\tau}}: w|v'} [(L_{\tau})_w:  k_v] \  - \lambda_{\tau} n \ \sum_{w' \in \Omega_{F_{\tau}}: w'|v'} [(F_{\tau})_{w'}:k_v]\right)\\
=&    \sum_{K \in S} \sum_{\tau \in T_K}   \left( [L_{\tau}:K]   - \lambda_{\tau} n  [F_{\tau}:K] \right) \sum_{v' \in \Omega_{K} : v'|v} [(K)_{v'}:k_v]\\
=&    \sum_{K \in S} [K:k] \sum_{\tau \in T_K}   \left( [L_{\tau}:K]   - \lambda_{\tau} n  [F_{\tau}:K] \right) \\
=&    \sum_{K \in S} [K:k] \sum_{\tau \in T_K}  \Delta_{\tau} \\
=&    d.
\end{array}\]
We claim that $(u_v)_v \in Z_0^d(X_{\A_k})^{\Br_1}$. Indeed, for any $\alpha \in \Br_1(X)$,
\[ 
\begin{array}{rl}
& \sum_{v \in \Omega_k}  \sum_{K \in S} \sum_{\tau \in T_K}\sum_{v' \in \Omega_{K} : v'|v}   \left( \sum_{w \in \Omega_{L_{\tau}}: w|v'} \inv_v\left(\cores_{(L_{\tau})_w/k_v} \left(\alpha \left( {x}^{(\tau)}_w \right) \right)\right) \right. \ \\
& \left.   - \lambda_{\tau}n \ \sum_{w' \in \Omega_{F_{\tau}}: w'|v'} \inv_v\left(\cores_{(F_{\tau})_{w'}/k_v} \left(\alpha \left(  \tilde{{x}}^{(\tau)}_{w'} \right) \right)\right)\right)\\
= &  \left(\sum_{v \in \Omega_k}  \sum_{K \in S} \sum_{\tau \in T_K}\sum_{v' \in \Omega_{K} : v'|v}    \sum_{w \in \Omega_{L_{\tau}}: w|v'} \inv_v\left(\cores_{(L_{\tau})_w/k_v} \left(\alpha \left( {x}^{(\tau)}_w \right) \right)\right) \right) \ \\
&   -\lambda_{\tau} n\left(  \sum_v  \sum_{K \in S} \sum_{\tau \in T_K}\sum_{v' \in \Omega_{K} : v'|v}   \ \sum_{w' \in \Omega_{F_{\tau}}: w'|v'} \inv_v\left(\cores_{(F_{\tau})_{w'}/k_v} \left(\alpha \left(  \tilde{{x}}^{(\tau)}_{w'} \right) \right)\right)\right)\\
\end{array}\]
Now,
\[ 
\begin{array}{rl}
 &  \sum_{v \in \Omega_k}  \sum_{K \in S} \sum_{\tau \in T_K}\sum_{v' \in \Omega_{K} : v'|v}    \sum_{w \in \Omega_{L_{\tau}}: w|v'} \inv_v\left(\cores_{(L_{\tau})_w/k_v} \left(\alpha \left( {x}^{(\tau)}_w \right) \right)\right) \\
=  &  \sum_{v \in \Omega_k}  \sum_{K \in S} \sum_{\tau \in T_K}\sum_{v' \in \Omega_{K} : v'|v}    \sum_{w \in \Omega_{L_{\tau}}: w|v'} \inv_w \left(\alpha \left( {x}^{(\tau)}_w \right) \right) \\
=  &   \sum_{K \in S} \sum_{\tau \in T_K}  \sum_{w \in \Omega_{L_{\tau}}} \inv_w \left(\alpha \left( {x}^{(\tau)}_w \right) \right) \\
=  &   \sum_{K \in S} \sum_{\tau \in T_K}  0 \\
= & 0,
\end{array}\]
where in the first equality we have used the commutative diagram 
\[\begin{array}{ccc}
\Br(\Spec((L_{\tau})_w))& \xrightarrow{\inv_w} & \Q/\Z\\  
 \xdownarrow{\cores_{(L_{\tau})_w/k_v}} & & \xdownarrow{=} \\
  \Br(\Spec(k_v))& \xrightarrow{\inv_v} & \Q/\Z\\  
\end{array}\]
and in third equality we have used the fact that $({x}^{\tau}_w)_{w \in \Omega_{L_\tau}} \in X(\A_{L_{\tau}})^{\Br_1(X_{L_\tau})} $.
Similarly, one can show that 
\[   \sum_{v \in \Omega_k}  \sum_{K \in S} \sum_{\tau \in T_K}\sum_{v' \in \Omega_{K} : v'|v}   \ \sum_{w' \in \Omega_{F_{\tau}}: w'|v'} \inv_v\left(\cores_{(F_{\tau})_{w'}/k_v} \left(\alpha \left(  \tilde{{x}}^{(\tau)}_{w'} \right) \right)\right) = 0.\]
Hence, $(u_v)_v \in Z_0^d(X_{\A_k})^{\Br_1}$, as required. 

Finally, we remark that, since $\sum_{w \in \Omega_{L_\tau}: w| v'} r_{L_{\tau, w}/K_{v'}}(m_{w})$ and $z_{v'}^\tau$ have the same image in $h_0(W^\tau_{K_{v'}})/n$ for all $\tau \in T_K$, all $K \in S$, and all places $v'$ above the places $v \in S'$, and by the compatibility of Suslin's homology with pushforward morphisms, it follows that   $\sum_{K \in S} \sum_{\tau \in T_K}\sum_{v' \in \Omega_{K} : v'|v}   \sum_{w \in \Omega_{L_{\tau}}: w|v'} s_{L_{\tau, w}/k_v} ({x}^{(\tau)}_w)$ and $z_v$ have the same image in $h_0(X_{k_v})/n$ for all $v \in S'$, and thus that $u_v$ and $z_v$ also have the same image in $h_0(X_{k_v})/n$ for all $v \in S'$. Since $X$ is proper, we note that $h_0(X_{k_v}) = \CH_0(X_{k_v})$. \\

(2) The assumption that $\Br_1 X/\Br_0 X$ is finite can only be true when $\Pic \Xbar$ is torsion-free: since $H^3_{\et}(k, \kbar^\times) = 0$ for any number field $k$ and $\kbar^\times[X] = \kbar^\times$ by our assumptions on $X$, by the long exact sequence coming from the Hochschild-Serre spectral sequence
\[ E_2^{p,q} := H^p_{\et}(k, H^q_{\et}(\Xbar, \G_m)) \Rightarrow H^{p+q}_{\et}(X, \G_m) \]

we know that
\[\Br_1 X/\Br_0 X = H^1_{\et}(k, \Pic \Xbar),\]
and, by Kummer theory, $H^1_{\et}(k, \Pic \Xbar)$ is infinite as soon as $\Pic \Xbar$ has non-trivial torsion.
Let $K/k$ be a finite Galois extension such that $\Gal(\kbar/K)$ acts trivially on $\Pic \Xbar$. Then, following for example the proof of \cite{Lia13}*{Prop 3.1.1}, we have that, for any finite extension $l/k$ linearly disjoint from $K$ over $k$, the natural restriction map
\[ \res_{l/k}: \Br_1 X/\Br_0 X \to \Br_1(X_l)/\Br_0(X_l)\]
is an isomorphism.

Let  $(z_v)_v \in \Z_0^d(X_{\A_k})^{\Br_1}$. We fix a closed point $x \in X$ of degree, say, $\delta_x:= [k(x):k]$. In particular, $((x)_w)_{w \in \Omega_{k(x)}} \in X(\A_{k(x)})^{\Br_1(X_{k(x)}) }$. Since $g_{k(x)} : W_{k(x)} \to X_{k(x)}$ is still a universal torsor, Theorem \ref{univtpt} yields $((x)_w)_{w \in \Omega_{k(x)}} \in X(\A_{k(x)})^{g_{k(x)}}$. 

Following \cite{Lia13}, we can construct a finite extension $l/k$, linearly disjoint from $K$ and $k(x)$ over $k$, such that \( [l:k] \equiv d \ (\bmod \ n \delta_x)\), say $[l:k] = d + \lambda n \delta_x$ for some $\lambda \in \Z$, and an adelic point $( \tilde{x}_{w'})_{w' \in \Omega_{l}} \in X(\A_l)^{\Br_1(X_l)}   = X(\A_l)^{g_l}$ with $\sum_{w' \in \Omega_{l}: w'| v} s_{l_{w'}/k_{v}}( \tilde{x}_{w'})$ sufficiently close (in the sense of \cite{Lia13}) to $z_{v}$ for each place $v \in S'$, where $s_{l_{w'}/k_{v}}: X_{l_{w'}} \to X_{K_{v'}}$ is the natural map. Note that $X(\A_l)^{\Br_1(X_l)}   = X(\A_l)^{g_l}$ by the fact that $g_l: W_l \to X_l$ is still a universal torsor together with Theorem \ref{univtpt}.

Consider the adelic 0-cycle on $X$ whose $v$-adic component is given by
\[ u_v := \sum_{w' \in \Omega_{l}: w'| v} s_{l_{w'}/k_{v}}( \tilde{x}_{w'}) - \lambda n \sum_{w \in \Omega_{k(x): w|v}} s_{k(x)_{w}/k_{v}}( (x)_w) \]
and has degree
\[ 
\begin{array}{rl}
\deg(z_v) =& \sum_{w' \in \Omega_{l}: w'| v} [l_{w'}: k_v] - \lambda n \sum_{w \in \Omega_{k(x): w|v}} [k(x)_w: k_v]\\
=& [l:k] - \lambda n  [k(x): k]\\
= & d.
\end{array}\]
We claim that $(u_v) \in Z_0^d(X_{\A_k})^g$. Indeed, take $S := \{l, k(x)\}$ (noticing that $l$ and $k(x)$ are linearly disjoint over $k$ by construction); take $T_{l} := \{\tau\} \subset H^1(l, G)$ and $T_{k(x)} := \{\sigma\} \subset H^1(k(x), G)$, where $\tau$ is any twist such that there exists some $(\tilde{y}_{w'})_{w'} \in W^\tau(\A_l)$ above $(\tilde{x}_{w'})_{w'}$,  and $\sigma$ is any twist such that there exists some  $(y_w)_w \in W^\sigma(\A_{k(x)})$ above $((x)_w)_w$; take $\Delta_\tau = 1$ and $ \Delta_\sigma = -\lambda n$. Then 
$(\tilde{y}_{w'})_{w'} \in Z_0^1(W^\tau_{\A_l}) $ and $(- \lambda n y_w)_w \in Z_0^{-\lambda}(W^\sigma_{\A_{k(x)}})$, with 
\[\textbf{g}_{\ast, \textrm{ad}}\left( \left((\tilde{y}_{w'})_{w'}, (-\lambda n y_{w})_{w}\right) \right) = \left( \sum_{w' \in \Omega_{l}: w'| v}  s_{l_{w'}/k_{v}}( \tilde{x}_{w'})  - \lambda n \sum_{w \in \Omega_{k(x): w|v}}  s_{k(x)_{w}/k_{v}}( (x)_w) \right)_v = (u_v)_v\]
and 
\[ \Delta_\tau \cdot [l:k] + \Delta_\sigma \cdot \delta_x = [l:k] - \lambda n \delta_x = d,\]
as required. Finally, it is easy to check that the $u_v$ and $z_v$ have the same image in $\CH_0(X_{k_v})/n$ for all $v \in S'$.
\end{proof}

Finally, we consider torsors under tori. In the rational points setting, Harpaz and Wittenberg have proved the following nice result.

\begin{thm}[{\cite{HW}*{Th\'eor\`eme 2.1}}] \label{hw} Let $X$ be a smooth, geometrically integral variety over $k$. Let $f: Y \to X$ be a torsor under a $k$-torus $T$. Let $A \subset \Br X$ be the inverse image of $\Br_{nr}(Y) \subset \Br Y$ under $f^\ast : \Br X \to \Br Y$. Then
\[ X(\A_k)^A \subset \bigcup_{\sigma \in H^1(k, T)} f^\sigma(Y^\sigma(\A_k)^{\Br_{nr}}).\]

\end{thm}

A very similar proof to that of Theorem \ref{univt} (2), together with Theorem \ref{hw},  yields the following immediate statement for 0-cycles in the spirit of Theorem \ref{hw}, albeit with the usual restrictions on the Brauer groups.

\begin{thm}\label{hw0cyc}
Let $X$ be a smooth, proper, geometrically integral variety over $k$. Let $f: Y \to X$ be a torsor under a $k$-torus $T$. Assume that $\Br X/\Br_0 X$ is finite and that there is some finite extension $F/k$ such that $\res_{l/k} : \Br X/ \Br_0 X \to \Br_{nr}(X_l)/ \Br_0(X_l)$ is surjective for all finite extensions $l/k$ linearly disjoint from $F$ over $k$. Then, for any integer $d \in \Z$, for any positive integer $n$, and for any finite subset $S' \subset \Omega_k$ of places of $k$, we have that  $(z_v)_v \in \Z_0^d(X_{\A_k})^{\Br_{nr}}$ implies that there exists some  $(u_v)_v \in \Z_0^d(X_{\A_k})^{f, \Br_{nr}}$ such that $z_v$ and $u_v$ have the same image in $\CH_0(X_{k_v})/n$ for all $v \in S'$.
\end{thm}

\noindent\emph{Acknowledgements.} The authors would like to thank David Harari,  Olivier Wittenberg, and the anonymous referee for insightful comments on a draft of this paper. They are also immensely grateful to Olivier Wittenberg for providing a proof of Proposition \ref{prop: Wittenberg} and Proposition \ref{prop: Olivier2}. During part of this work, FB was supported by the European Union’s Horizon 2020 research and  
programme under the Marie Sklodowska-Curie grant 840684.

\bibliographystyle{abbrv}
\bibliography{refs}
\end{document}